\documentclass[11pt,reqno,a4paper]{amsart}

\usepackage{a4wide}
\usepackage{amssymb}
\usepackage{amsfonts}
\usepackage{enumerate}
\usepackage{amsmath}
\usepackage{bbm}
\usepackage{ stmaryrd }

\newtheorem{theorem}{Theorem}[section]

\newtheorem{definition}[theorem]{Definition}
\newtheorem{example}[theorem]{Example}

\newtheorem{lemma}[theorem]{Lemma}

\newtheorem{proposition}[theorem]{Proposition}
\newtheorem{remark}[theorem]{Remark}

\def\eps{\varepsilon}

\def\N{\mathbb{N}}

\def\R{\mathbb{R}}
\def\P{\mathbb{P}}
\def\E{\mathbb{E}}

\def\F{\mathcal{F}}

\DeclareMathOperator{\var}{var}

\def\lt{\left}
\def\rt{\right}

\begin{document}

\title[Shortest distance between multiple orbits]{Shortest distance between multiple orbits and generalized fractal dimensions}
\author{Vanessa Barros, J\'er\^ome Rousseau}
\address{Vanessa Barros and J\'er\^ome Rousseau, Departamento de Matem\'atica, Universidade Federal da Bahia\\
Av. Ademar de Barros s/n, 40170-110 Salvador, Brazil}
\address{Departamento de Matem\'atica, Faculdade de Ci\^encias da Universidade do Porto,\\Rua do Campo Alegre, 687, 4169-007 Porto, Portugal}
\email{vbarrosoliveira@gmail.com}
\urladdr{https://sites.google.com/site/vbarrosoliveira/home}
\email{jerome.rousseau@ufba.br}
\urladdr{http://www.sd.mat.ufba.br/~jerome.rousseau}

\thanks{This work was partially supported by CNPq, by FCT project PTDC/MAT-PUR/28177/2017, with national funds, and by CMUP (UID/MAT/00144/2019), which is funded by FCT with national (MCTES) and European structural funds through the programs FEDER, under the partnership agreement PT2020.}

\maketitle

\begin{abstract}
We consider rapidly mixing dynamical systems and link the decay of the shortest distance between multiple orbits with the generalized fractal dimension. We apply this result to multidimensional expanding maps and extend it to the realm of random dynamical systems. For random sequences, we obtain a relation between the longest common substring between multiple sequences and the generalized R\'enyi entropy. Applications to Markov chains, Gibbs states and the stochastic scrabble are given.
\end{abstract}

\section{Introduction}
Generalized fractal dimensions were originally introduced to characterize and measure the strangeness of chaotic attractors and, more generally, to describe the fractal structure of invariant sets in dynamical systems \cite{grass,gp1,gp2}.

Given $k>1$, the generalized fractal dimension (also known as $L^q$ or $HP$ dimensions) of a measure $\mu$ is defined (provided the limit exists) by:
\[{D}_k(\mu)=\underset{r\rightarrow0}{\lim}\frac{\log\int_X\mu\left(B\left(x,r\right)\right)^{k-1}d\mu(x)}{(k-1)\log r}.\]
For the existence of these dimensions, their properties and relations with other dimensions, one can see e.g. \cite{barbaroux,Fan,Pes, PW}.

Since estimation of the generalized dimensions plays an important role in the description of dynamical systems, different numerical approaches and procedures have been developed to compute them (see e.g. \cite{alber, badii1,badii2,vaienti-gen, caby-vaienti,pastor} and references within). In particular, we highlight \cite{faranda-vaienti} where Extreme Value Theory (EVT) was used as a tool to estimate the correlation dimension  ${D}_2(\mu)$, and \cite{caby-vaienti} for generalized dimensions. For a deeper discussion of EVT for dynamical systems we refer the reader to \cite{book-evt}.

It is also worth mentioning the connection between generalized dimensions and the recurrence properties of the dynamics.
Return time dimensions and generalized fractal dimensions were thoroughly compared in \cite{luevano,mantica}. Moreover, they appear in the rate function for the large deviations of the return time \cite{ caby-vaienti,largedev}.

In this communication we study, for a dynamical system $(X,T,\mu)$, the behaviour of the shortest distance between $k$ orbits, i.e. for $(x_1,\dots,x_k)\in X^k$:
\begin{equation}\label{mn}m_n(x_1,\dots,x_k)=\min_{i_1,\dots,i_k=0,\dots,n-1}\left(d(T^{i_1}x_1,\dots,T^{i_k}x_k)\right),\end{equation}
where $d(x_1,\dots,x_k)=\underset{i\neq j}{\max} \ d(x_i,x_j)$, and show a relation between this shortest distance and the generalized fractal dimensions.

Indeed, if the generalized dimension exists, then under some rapid mixing conditions on the system $(X,T,\mu)$, for $\mu\otimes\dots\otimes\mu$-almost every $(x_1,\dots,x_k)\in X^k$, we have

\begin{equation}\label{a}
\underset{n\rightarrow+\infty}{\lim}\frac{\log m_n(x_1,\dots,x_k)}{-\log n}=\frac{k}{(k-1){D}_k(\mu)}.
\end{equation}
In particular, we apply these results to the multidimensional expanding maps defined by Saussol \cite{Saussol1}. Moreover, we also prove an annealed version of \eqref{a} for the shortest distance between $k$ orbits of a random dynamical system.

These results extend and complement those in  \cite{BaLiRo} and \cite{CoLaRo} where identity \eqref{a} (and its equivalent for random dynamical systems) was proved for two orbits ($k=2$).

Furthermore, it was shown in \cite{BaLiRo} that the problem of the shortest distance between orbits is a generalization of the longest common substring problem for random sequences, a problem thoroughly investigated in genetics, probability and computer science (see e.g \cite{W-book}). More precisely,  for $\alpha$-mixing systems, they study the behaviour of the length of the longest common substring between two sequences $x$ and $y$:
\[M_n(x,y)=\max\{m:x_{i+k}=y_{j+k}\textrm{ for $k=1,\dots,m$ and for some $0\leq i,j\leq n-m$}\},\]
and generalized the work of Arratia and Waterman \cite{AW} where only independent irreducible and aperiodic Markov chains on a finite alphabet were considered. More recently, similar results for encoded sequences \cite{CoLaRo}, random sequences in random environment \cite{LCS-random}, stationary determinantal process on the integer lattice \cite{FLQ} and also for the longest matching consecutive subsequence between two $N$-ary expansions \cite{LiYang} were obtained.


Following the ideas in \cite{BaLiRo}, we extend here our study to the longest common substring between multiple sequences (previous results in this direction were obtained in \cite{KO,KO1}). More precisely, for $k$ sequences $x^1,\dots, x^k$, we define the length of the longest common substring by
\begin{align*}
M_n&(x^1,\dots, x^k)\\
&=\max\{m:x^1_{i_1+j}=...=x^k_{i_k+j}\textrm{ for $j=0,...,m-1$ and for some $0\leq i_1,...,i_k\leq n-m$}\}.
\end{align*}
and link it to the generalized R\'enyi entropy (provided that it exists, see e.g. \cite{AbCa, AbLa2,haydn-vaienti,luc}):
\[{H}_k=\underset{n\rightarrow+\infty}{\lim}\frac{\log\sum \P(C_n)^k}{-(k-1)n},\]
where the sum is taken over all $n$-cylinders $C_n$ (see Section \ref{sec-lcs} for a precise definition).

Thus, we prove that for $\alpha$-mixing systems with exponential decay (and $\psi$-mixing with polynomial decay), if the generalized R\'enyi entropy exists, then for $\P^k$-almost every $(x^1,\dots,x^k)$,
\[ \underset{n\rightarrow+\infty}{\lim}\frac{M_n(x^1,\dots,x^k)}{\log n}=\frac{k}{(k-1)H_k}.\]
Moreover, we also prove a version of this result for encoded sequences.

The paper is organized as follows. Our main results linking the shortest distance between multiple orbits and the generalized fractal dimensions are stated in Section~\ref{secshort} and proved in Section~\ref{sec-proof}. An application of these results for multidimensional expanding maps is given in Section~\ref{secexp}. Shortest distance between multiple observed orbits and random orbits are studied in Section~\ref{secrandom}.
 In Section~\ref{sec-lcs}, we study the longest common substring problem for multiple random sequences (and encoded sequences) and its relation with the generalized R\'enyi entropy. These results are proved in Section~\ref{sec-discrete}. 

\section{Shortest distance between $k$ orbits}\label{secshort} 

Let $(X,d)$ be a finite dimensional metric space and $\mathcal{A}$ its Borel $\sigma$-algebra.
Let $(X,\mathcal{A},\mu,T)$ be a measure preserving system which means that $T:X\rightarrow X$ is a transformation on $X$ and $\mu$ is a probability measure on $(X,\mathcal{A})$ such that $\mu$ is invariant by $T$, i.e., $\mu(T^{-1}A)=\mu(A)$ for all $A\in\mathcal{A}$. We will denote by $\mu^k$ the product measure $\mu\otimes\dots\otimes\mu$.

We would like to study the behaviour of the shortest distance between $k$ orbits:
\[m_n(x_1,\dots,x_k)=\min_{i_1,\dots,i_k=0,\dots,n-1}\left(d(T^{i_1}x_1,\dots,T^{i_k}x_k)\right)\]
where $d(x_1,\dots,x_k)=\max_{i\neq j} d(x_i,x_j)$.
\begin{remark} Other definitions could have been chosen for $d(x_1,\dots,x_k)$ without altering our results (see e.g. \cite{joly,nway} and references therein for examples of generalizations of the usual two-way distance).
For example, we could have used  $d_1(x_1,\dots,x_k)=\min_{z\in X}\max_{i} d(x_i,z)$, or  $d_2(x_1,\dots,x_k)=\sqrt{\sum_{i\neq j}d(x_i,x_j)^2}$ but our results would have been the same since $d,\ d_1,$ and $d_2$ are equivalent.
\end{remark}

We will show that the behaviour of $m_n$ as $n\rightarrow\infty$ is linked with the generalized fractal dimension. Before stating the first theorem, we recall, for $k>1$, the definition of the lower and upper generalized fractal dimensions of $\mu$:
\[\underline{D}_k(\mu)=\underset{r\rightarrow0}{\underline\lim}\frac{\log\int_X\mu\left(B\left(x,r\right)\right)^{k-1}d\mu(x)}{(k-1)\log r}\qquad\textrm{and}\qquad\overline{D}_k(\mu)=\underset{r\rightarrow0}{\overline\lim}\frac{\log\int_X\mu\left(B\left(x,r\right)\right)^{k-1}d\mu(x)}{(k-1)\log r}.\]
When the limit exists we will denote the common value of $\underline{D}_k(\mu)$ and $\overline{D}_k(\mu)$ by ${D}_k(\mu)$. 

\begin{theorem}\label{thineq}

Let $(X,\mathcal{A},\mu,T)$ be a measure preserving system such that $\underline{D}_k(\mu)>0$. Then for $\mu^k$-almost every $(x_1,\dots, x_k)\in X^k$,
\[ \underset{n\rightarrow+\infty}{\overline\lim}\frac{\log m_n(x_1,\dots,x_k)}{-\log n}\leq\frac{k}{(k-1)\underline{D}_k(\mu)}.\]

\end{theorem}

This general result can be applied to any dynamical system such that $\underline{D}_k(\mu)>0$. Even if the inequality in Theorem \ref{thineq} can be strict (noting for example the trivial case when $T$ is the identity), we will prove that an equality holds under some rapidly mixing conditions:

(H1) There exists a Banach space $\mathcal{C}$, such that for all $\psi,\ \phi\in \mathcal{C}$ and for all $n\in\N^*$,  we have
\[\left|\int_X\psi.\phi\circ T^n\, d\mu-\int_X \psi d\mu\int_X\phi d \mu\right|\leq\|\psi\|_\mathcal{C}\|\phi\|_\mathcal{C}\theta_n,\]
with $\theta_n =a^{n}$ ($0\leq a<1$) and where $\|\cdot\|_\mathcal{C}$ is the norm in the Banach space $\mathcal{C}$.

(H2) There exist $0<r_0<1$, $c\geq0$ and $\xi \geq0$ such that for every $p\in\{1,\dots,k\}$, for $\mu^{k-p}$-almost every $x_{p+1},\dots,x_k\in X$ and any $0<r<r_0$, the function $\psi_p:X\rightarrow \R$, defined below, belongs to the Banach space $\mathcal{C}$ and verify
\[\|\psi_p\|_\mathcal{C}\leq cr^{-\xi}.\]

Fixed $x_2,\dots, x_k \in X$, we define
\begin{equation}\label{psi1}\psi_1(x)= \prod_{j=2}^k \mathbbm{1}_{B(x_j,r)}(x). \end{equation}

For $p>1$, we fix $x_{p+1},\dots,x_k \in X$, and set
\begin{equation}\label{psi2}\psi_p(x)=\bar\psi(x,x_{p+1},\dots,x_k), \text{ where}\end{equation}
\begin{align*}\bar\psi(&x_p,x_{p+1},\dots,x_k)\\
 &=\prod_{l=p+1}^k \mathbbm{1}_{B(x_l,r)}(x_p) \int_{X^{p-1}}\left[\prod_{j=1}^{p-1}\prod_{l=j+1}^{k}  \mathbbm{1}_{B(x_j,r)}(x_l)\right]d\mu^{p-1}(x_1,\dots,x_{p-1}).\end{align*}

When the Banach space $\mathcal{C}$ is the space of H\"older functions $\mathcal{H}^\alpha(X,\R)$, we will replace our assumption (H2) by an assumption easier to interpret in Theorem~\ref{thdsuphold}.

We will also need some topological information on the space $X$.
\begin{definition}\label{deftight}
A separable metric space $(X,d)$ is called tight if there exist $r_0>0$ and $N_0\in \N$, such that for any $0 < r < r_0$ and any $x\in X$ one can cover $B(x, 2r)$ by at most $N_0$ balls of radius $r$.
\end{definition}
We emphasize that any subset of $\R^n$ with the Euclidian metric is tight, any subset of a Riemannian manifold of bounded curvature is tight and that if $(X,d)$ admits a doubling measure then it is tight \cite{GY}. 

Now we can state our main result.
\begin{theorem}\label{thprinc}
Let $(X,\mathcal{A},\mu,T)$ be a measure preserving system, such that $(X,d)$ is tight, satisfying (H1) and (H2) and such that ${D}_k(\mu)$ exists and is strictly positive. Then for $\mu^k$-almost every $(x_1,\dots,x_k)\in X^k$,
\[ \underset{n\rightarrow+\infty}{\lim}\frac{\log m_n(x_1,\dots,x_k)}{-\log n}=\frac{k}{(k-1){D}_k(\mu)}.\]
\end{theorem}
Now, we will apply this result to a short list of simple examples. Later, in Section~\ref{secexp}, we use this theorem for a more complex family of examples (multidimensional piecewise expanding maps).

\smallskip
Denote by $Leb$ the Lebesgue measure.
\begin{example}
Theorem~\ref{thprinc} can be applied to the following systems:
\begin{enumerate}
\item For $m\in\{2,3,\dots\}$, let $T: [0,1] \rightarrow [0,1]$ be such that $x\mapsto mx \mod 1$ and $\mu=Leb$.
\item Let $T:(0,1]\rightarrow(0,1]$ be such that $T(x)=2^n(x-2^{-n})$ for $x\in(2^{-n},2^{-n+1}]$ and $\mu=Leb$.
\item ($\beta$-transformations) For $\beta>1$, let $T: [0,1] \rightarrow [0,1]$ be such that $x\mapsto \beta x \mod 1$ and $\mu$ be the Parry measure (see  \cite{parry}), which is an absolutely continuous probability measure with density $\rho$ satisfying $1-\frac{1}{\beta}\leq\rho(x)\leq(1-\frac{1}{\beta})^{-1}$ for all $x\in[0,1]$. 
\item (Gauss map) Let $T:(0,1]\rightarrow(0,1]$ be such that $T(x)=\left\{\frac{1}{x}\right\}$ and $d\mu=\frac{1}{\log 2}\frac{dx}{1+x}$.
\end{enumerate}
In these examples it is easy to see that $D_k(\mu)=1$. Moreover, (H1) and (H2) are satisfied with the Banach space $\mathcal{C}=BV,$ the space of functions of bounded variation (see e.g. \cite{FFT} Section 4.1 and \cite{hof,philipp,philipp2}).
\end{example}

One can observe that Theorem \ref{thprinc} is an immediate consequence of Theorem \ref{thineq} and the next theorem.
\begin{theorem}\label{dsup}

Let $(X,\mathcal{A},\mu,T)$ be a measure preserving system, such that $\underline{D}_k(\mu)>0$ and such that $(X,d)$ is tight, satisfying (H1) and (H2). Then for $\mu^k$-almost every $(x_1,\dots,x_k)\in X^k$,
\[ \underset{n\rightarrow+\infty}{\underline\lim}\frac{\log m_n(x_1,\dots,x_k)}{-\log n}\geq\frac{k}{(k-1)\overline{D}_k(\mu)}.\]

\end{theorem}

When the Banach space $\mathcal{C}$ is the space of H\"older functions $\mathcal{H}^\alpha(X,\R)$ we can adapt our proof and (H2) can be replaced by the following condition:

(HA) There exist $r_0>0$, $\xi\geq0$ and $\beta>0$ such that for $\mu$-almost every $x\in X$ and any $r_0>r>\rho>0$,
\[\mu(B(x,r+\rho)\backslash B(x,r-\rho))\leq r^{-\xi}\rho^\beta.\]
 
This assumption is satisfied, for example, if the measure is Lebesgue or absolutely continuous with respect to Lebesgue with a bounded density. 

\begin{theorem}\label{thdsuphold}
Let $(X,\mathcal{A},\mu,T)$ be a measure preserving system, such that $\underline{D}_k(\mu)>0$ and such that $(X,d)$ is tight, satisfying (H1) with $\mathcal{C}=\mathcal{H}^\alpha(X,\R)$ and (HA). Then for $\mu^k$-almost every $(x_1,\dots,x_k)\in X^k$,
\[ \underset{n\rightarrow+\infty}{\underline\lim}\frac{\log m_n(x_1,\dots,x_k)}{-\log n}\geq\frac{k}{(k-1)\overline{D}_k(\mu)}.\]
\end{theorem}

For example, one can apply this theorem to expanding maps of the interval with a Gibbs measure associated to a H\"older potential (see e.g. \cite{saussol}) and $C^2$ endomorphism (of a $d$-dimensional compact Riemannian manifold) admitting a Young tower with exponential tail (see \cite[Section 6]{FFT10}  and \cite{collet}).


\section{Observed orbits and random dynamical systems}\label{secrandom}


In this section, we extend our analysis to the study of observation of orbits. Indeed, considering observations of systems (for example, temperature or pressure while studying climate) could be more significant than considering the whole system. From a more theoretical point of view, we will explain in Section~\ref{subsecrandom} how the study of observed orbits allows us to study random dynamical systems. 

Let $(Y,d)$ be a metric space and $f:X\rightarrow Y$ be a measurable function (called the observation). We denote by $f_*\mu$ the pushforward measure, defined by $f_*\mu(A)=\mu(f^{-1}(A))$ for measurable subsets $A\subset Y$.

We would like to study the behaviour of the shortest distance between $k$ observed orbits:
\[m_n^f(x_1,\dots,x_k)=\min_{i_1,\dots,i_k=0,\dots,n-1}\left(d(f(T^{i_1}x_1),\dots,f(T^{i_k}x_k))\right).\]

\begin{theorem}\label{thineqobs}

Let $(X,\mathcal{A},\mu,T)$ be a measure preserving system such that $\underline{D}_k(f_*\mu)>0$. Then for $\mu^k$-almost every $(x_1,\dots, x_k)\in X^k$,
\[ \underset{n\rightarrow+\infty}{\overline\lim}\frac{\log m_n^f(x_1,\dots,x_k)}{-\log n}\leq\frac{k}{(k-1)\underline{D}_k(f_*\mu)}.\]

\end{theorem}

We will assume that $f$ is Lipschitz and as in Section~\ref{secshort}, we  prove that the equality holds under some rapidly mixing conditions: 

(H1') For all $\psi,\ \phi\in \mathcal{H}^\alpha(Y,\R)$ and for all $n\in\N^*$,  we have
\[\left|\int_X\psi(f(x)).\phi (f(T^nx))\, d\mu(x)-\int_X \psi(f(x)) d\mu(x)\int_X\phi(f(x)) d \mu(x)\right|\leq\|\psi\circ f\|_{\mathcal{H}^\alpha}\|\phi\circ f\|_{\mathcal{H}^\alpha}\theta_n,\]
with $\theta_n =a^{n}$ ($0\leq a<1$).

For simplicity, we only treat the case when the mixing property is satisfied for H\"older observables. However, we observe that on can adapt (H1) and (H2) to this setting to work with other Banach spaces.

Now we can state our version of Theorem \ref{thdsuphold} for observed orbits.
\begin{theorem}\label{thprincobs}
Let $(X,\mathcal{A},\mu,T)$ be a measure preserving system and $f$ a Lipschitz observation, such that $\underline{D}_k(f_*\mu)>0$ and such that $(Y,d)$ is tight, satisfying (H1') and such that $f_*\mu$ satisfies (HA). Then for $\mu^k$-almost every $(x_1,\dots,x_k)\in X^k$,
\[ \underset{n\rightarrow+\infty}{\underline\lim}\frac{\log m_n^f(x_1,\dots,x_k)}{-\log n}\geq\frac{k}{(k-1)\overline{D}_k(f_*\mu)}.\]
Moreover, if ${D}_k(f_*\mu)$ exists, then for $\mu^k$-almost every $(x_1,\dots,x_k)\in X^k$,
\[ \underset{n\rightarrow+\infty}{\lim}\frac{\log m_n^f(x_1,\dots,x_k)}{-\log n}=\frac{k}{(k-1){D}_k(f_*\mu)}.\]

\end{theorem}


\subsection{Shortest distance between multiple random orbits}\label{subsecrandom}
In this subsection, we will use the previous results to study the shortest distance between multiple orbits of a random dynamical system.

Let $(X,d)$ be a tight metric space and let $(\Omega, \theta, \mathbb{P})$ be a probability measure preserving system, where $\Omega$ is a metric space and $B(\Omega)$ its Borelian $\sigma$-algebra. 

\begin{definition}
A random dynamical system $\mathcal{T}=(T_{\omega})_{\omega \in \Omega}$ on $X$ over $(\Omega, B(\Omega), \mathbb{P}, \theta)$ is generated by maps $T_\omega$ such that $(\omega,x) \mapsto T_\omega(x)$ is measurable and satisfies:
$$T_\omega^0=Id \ \mbox{for all} \ \omega \in \Omega,$$
$$T_\omega^n= T_{\theta^{n-1}(\omega)}\circ \cdots \circ T_{\theta(\omega)} \circ T_\omega \ \mbox{for all} \ n \geq 1.$$
The map $S: \Omega \times X \to \Omega \times X$ defined by $S(\omega, x)=(\theta(\omega),T_\omega(x))$ is the dynamics of the random dynamical systems generated by $\mathcal{T}$ and is called skew-product.

A probability measure $\mu$ is said to be an invariant measure for the random dynamical system $\mathcal{T}$ if it satisfies

\begin{itemize}
\item [1.] $\mu$ is $S$-invariant
\item [2.] $\pi_*\mu= \mathbb{P}$
\end{itemize}
where $\pi: \Omega \times X \to \Omega$ is the canonical projection.

Let $(\mu_\omega)_\omega$ denote the decomposition of $\mu$ on $X$, that is, $d\mu(\omega,x)=d\mu_\omega(x)d\mathbb{P}(\omega)$. We denote by $\nu=\int \mu_\omega d\mathbb{P}$ the marginal of $\nu$ on $X$.
\end{definition}

For $(\omega_1,x_1), \dots,(\omega_k,x_k)$, we define the shortest distance between $k$ random orbits by
$$m_n^{\omega_1,\dots,\omega_k}(x_1,\dots,x_k) =  \min_{i_1,\dots,i_k=0,\dots,n-1} \lt(d\lt(T_{\omega_1}^{i_1}(x_1),\dots,T_{{\omega_k}}^{i_k}({x_k})\rt)\rt).$$
\begin{remark}
We observe that the technic developed here only allows us to obtain annealed results. Another object worth studying would be the quenched shortest distance
$$m_n^{\omega}(x_1,\dots,x_k) =  \min_{i_1,\dots,i_k=0,\dots,n-1} \lt(d\lt(T_{\omega}^{i_1}(x_1),\dots,T_{{\omega}}^{i_k}({x_k})\rt)\rt).$$
In this direction, the only known results are for $2$ orbits and when the system is a random subshift of finite type \cite{LCS-random}.
\end{remark}

As in the deterministic case, we will assume an exponential decay of correlations for the random dynamical system:

(H1R) (Annealed decay of correlations) For every $n \in \mathbb{N}^*$, and every $\psi$, $\phi \in\mathcal{H}^\alpha(X,\R)$,
$$\lt|\int_{\Omega \times X} \psi (T^n_\omega(x))\phi(x) \ d\mu(\omega,x) - \int_{\Omega \times X} \psi \ d\mu \int_{\Omega \times X} \phi \ d\mu \rt| \leq \|\psi \|_{\mathcal{H}^\alpha}\|\phi \|_{\mathcal{H}^\alpha} \theta_n,$$
with $\theta_n=a^{n}$ ($0\leq a<1$).

\begin{theorem}\label{theoremrandom}
Let $\mathcal{T}$ be a random dynamical system on $X$ over $(\Omega, B(\Omega), \mathbb{P}, \theta)$ with an invariant measure $\mu$ such that $\underline{D}_k({\nu})>0$. Then for $\mu^k$-almost every $(\omega_1, x_1,\dots,{\omega_k},{x_k}) \in (\Omega \times X)^k,$
$$
\underset{n \rightarrow \infty}{\overline{\lim}}\frac{\log m_n^{\omega_1,\dots,\omega_k}(x_1,\dots,x_k)}{-\log n} \leq \frac{k}{(k-1)\underline{D}_k({\nu})}. \
$$
Moreover, if the random dynamical system satisfies assumptions $(H1R)$ and $\nu$ satisfies (HA), then
$$
\underset{n \rightarrow \infty}{\underline{\lim}}\frac{\log m_n^{\omega_1,\dots,\omega_k}(x_1,\dots,x_k)}{-\log n} \geq \frac{k}{(k-1)\overline{D}_k({\nu})},
$$
and if ${D}_k({\nu})$ exists, then
\begin{equation*}
  \underset{n \rightarrow \infty}{\lim}\frac{\log m_n^{\omega_1,\dots,\omega_k}(x_1,\dots,x_k)}{-\log n} = \frac{k}{(k-1)D_k({\nu})} \ .
\end{equation*}

\end{theorem}

\begin{proof}
Following the ideas in \cite{Rousseau}, it is enough to apply Theorem \ref{thineqobs} and Theorem \ref{thprincobs} for the dynamical system $(\Omega \times X, B(\Omega \times X), \mu, S)$ with the observation $f$ defined by
\begin{eqnarray*}
&& f: \Omega \times X \to X \\
&& \ \ \ \ \ \ (\omega,x) \mapsto x.
\end{eqnarray*}
\end{proof}

We now apply  the above result to some simple non-i.i.d. random dynamical system and we observe that, as in \cite{CoLaRo}, Theorem~\ref{theoremrandom} could also be applied to randomly perturbed dynamical systems and random hyperbolic toral automorphisms.

\begin{example}[Non-i.i.d. random expanding maps] 

Consider the two following linear maps
\begin{eqnarray*}
&& T_1:X \to X      \quad \mbox{and} \quad  T_2: X \to X \\
&& \ \qquad x \mapsto 2x \hspace{2.1cm} x \mapsto 3x,
\end{eqnarray*}
where $X$  is the  one-dimensional torus $ \mathbb{T}^1$. It is easy to see that $T_1$ and $T_2$ preserve the Lebesgue measure($Leb$).

The following skew product gives the dynamics of the random dynamical system:
\begin{eqnarray*}
&& S: \Omega \times X \to \Omega \times X \\
&& \ \ \ \ \ \ (\omega,x) \mapsto (\theta(\omega), T_\omega (x)),
\end{eqnarray*}
with $\Omega=[0,1], \ T_\omega=T_1$ if $\omega \in [0,2/5)$ and $T_\omega=T_2$ if $\omega \in [2/5,1]$ where $\omega$ is the following piecewise linear map:

$$\theta(\omega)=\left\{\begin{array}{ll}
2\omega &  \text{ if } \, \omega\in [0,1/5) \\
3\omega-1/5 &  \text{ if } \, \omega \in [1/5,2/5) \\
2\omega-4/5 &  \text{ if }\, \omega \in [2/5,3/5) \\
3\omega/2-1/2 &  \text{ if } \,\omega \in [3/5,1].
\end{array}\right.$$

The associated skew-product $S$ is $Leb \otimes Leb$-invariant. It is easy to check that Lebesgue measure satisfies (HA). Moreover, by \cite{Baladi} the skew product $S$ has an exponential decay of correlations. Since in this example $\nu=Leb$, we have $D_k(\nu)=1$ and Theorem \ref{theoremrandom} implies that for $Leb^{2k}$-almost every $(\omega_1, x_1,\dots,{\omega_k},{x_k}) \in ([0,1]\times  \mathbb{T}^1)^k $,

$$  \underset{n \rightarrow \infty}{\lim}\frac{\log m_n^{\omega_1,\dots,\omega_k}(x_1,\dots,x_k)}{-\log n} = \frac{k}{k-1}.$$

\end{example}

\section{Longest common substring between $k$ random sequences}\label{sec-lcs}
It was shown in \cite{BaLiRo} that studying the shortest distance between orbits for a symbolic dynamical system coincides with studying the length of the longest common substring between sequences.

Thus we will consider the symbolic dynamical systems $(\Omega,\P,\sigma)$, where $\Omega=\mathcal{A}^\N$ for some alphabet $\mathcal{A}$, $\sigma$ is the (left) shift on $\Omega$ and  $\P$ is a $\sigma$-invariant probability measure. For $k$ sequences $x^1,\dots, x^k\in\Omega$, we are interested in the behaviour of
\begin{align*}
M_n(&x^1,..., x^k)\\
&=\max\{m:x^1_{i_1+j}=...=x^k_{i_k+j}\textrm{ for $j=0,...,m-1$ and for some $0\leq i_1,...,i_k\leq n-m$}\}.
\end{align*}
We will show that the behaviour of $M_n$ is linked with the generalized R\'enyi entropy of the system. 

 For $y\in \Omega$ we denote by $C_n(y)=\{z \in \Omega :z_i = y_i\text{ for all } 
0\le i\le n-1\} $ the  \emph{$n$-cylinder} containing $y$. Set $\F_0^n$ as the sigma-algebra over $\Omega$ 
generated by all $n$-cylinders.

For $k>1$, we recall the definition of the lower and upper generalized R\'enyi entropy:
\[\underline{H}_k(\P)=\underset{n\rightarrow+\infty}{\underline\lim}\frac{\log\sum \P(C_n)^k}{-(k-1)n}\qquad\textrm{and}\qquad\overline{H}_k(\P)=\underset{n\rightarrow+\infty}{\overline\lim}\frac{\log\sum \P(C_n)^k}{-(k-1)n},\]
where the notation $\sum \P(C_n)^k$ means  $\underset{y\in \mathcal{A}^n}{\sum} \P(C_n(y))^k$. When the limit exists, we will denote it by $H_k(\P)$.

We say that a system $(\Omega,\P,\sigma)$ is {\it $\alpha$-mixing} if there exists
a function $\alpha:\N \rightarrow\R$ satisfying $\alpha(g)\to0$ when $g\to+\infty$ and such that
for all $m,n \in \N$, $A\in\F_0^n$ and $B\in \F_0^{m}$:
\[
\left|\P(A\cap\sigma^{-g-n}B) -\P(A)\P(B)\right|\le \alpha(g).
\]
It is said to be  {\it $\alpha$-mixing with an exponential decay} if the function $\alpha(g)$ decreases exponentially fast to $0$.

We say that our system is {\it $\psi$-mixing} if there exists
a function $\psi:\N \rightarrow\R$ satisfying $\psi(g)\to 0$ when $g\to+\infty$ and such that
for all $m,n \in \N$, $A\in\F_0^n$ and $B\in \F_0^{m}$:
\[
\left|\P(A\cap\sigma^{-g-n}B) -\P(A)\P(B)\right|\le \psi(g)\P(A)\P(B).
\]
Now we are ready to state our next result.

\begin{theorem}\label{seqmat}

If $\underline{H}_k(\P)>0$, then for $\P^k$-almost every $(x^1,\dots,x^k)\in \Omega^k$,
\begin{equation}\label{eqren1}
 \underset{n\rightarrow+\infty}{\overline\lim}\frac{M_n(x^1,\dots,x^k)}{\log n}\leq\frac{k}{(k-1)\underline{H}_k(\P)}.
 \end{equation}
Moreover, if the system is $\alpha$-mixing with an exponential decay or if it is $\psi$-mixing with $\psi(g)=g^{-a}$ for some $a>0$ then, for $\P^k$-almost every $(x^1,\dots,x^k)\in \Omega^k$,
\begin{equation}\label{eqren2} 
\underset{n\rightarrow+\infty}{\underline\lim}\frac{M_n(x^1,\dots,x^k)}{\log n}\geq\frac{k}{(k-1)\overline{H}_k(\P)}.
\end{equation}
Therefore, if the generalized R\'enyi entropy exists, then for $\P^k$-almost every $(x^1,\dots,x^k)\in \Omega^k$,
\[ \underset{n\rightarrow+\infty}{\lim}\frac{M_n(x^1,\dots,x^k)}{\log n}=\frac{k}{(k-1)H_k(\P)}.\]

\end{theorem}
This theorem can be applied, for example, to Markov chains and Gibbs states:
\begin{example}[Markov chains]
If $(\Omega,\P,\sigma)$ is an irreducible and aperiodic Markov chain on a finite alphabet $\mathcal{A}$, then it is $\psi$-mixing with an exponential decay (see e.g. \cite{Bradley}).
If we denote by $P$ the associated stochastic matrix (with entries $P_{ij}$), then the matrix $P(k)$ whose entries are $P_{ij}(k)=P_{ij}^{k}$ has, by the Perron-Frobenius theorem, a single largest eigenvalue $\lambda_k$. Moreover, the generalized R\'enyi entropy exists and $H_k(\P)=-\log \lambda_k/(k-1)$ \cite{haydn-vaienti}. Thus, for $\P^k$-almost every $(x^1,\dots,x^k)\in \Omega^k$,
\[ \underset{n\rightarrow+\infty}{\lim}\frac{M_n(x^1,\dots,x^k)}{\log n}=\frac{k}{-\log \lambda_k}.\]
\end{example}

\begin{example}[Gibbs states]
Let $\P$ be a Gibbs state of a H\"older-continuous potential $\phi$. Then, the system is $\psi$-mixing with an exponential decay \cite{Bowen, Ruelle}. Moreover, the generalized R\'enyi entropy exists and $H_k(\P)=(1/(k-1))\left(kP(\phi)-P(k\phi)\right)$ where $P(\phi)$ is the pressure of the potential $\phi$ \cite{haydn-vaienti}. Thus, for $\P^k$-almost every $(x^1,\dots,x^k)\in \Omega^k$,
\[ \underset{n\rightarrow+\infty}{\lim}\frac{M_n(x^1,\dots,x^k)}{\log n}=\frac{k}{kP(\phi)-P(k\phi)}.\]
\end{example}


\subsection{Encoded sequences}
While working with sequences, one can wonder if similar results to the one presented in Theorem~\ref{seqmat} are still satisfied if the original sequences are modified or encoded in some way (e.g. contaminated, compressed, decompressed).

 Thus, for a measurable function $f: \Omega \to \tilde\Omega $ (called {an encoder}) and for $k$ sequences $x^1,\dots, x^k\in\Omega$, we are interested in the behaviour of
\begin{align*}
&M_n^f(x^1,..., x^k)\\
&=\max\{m:f(x^1)_{i_1+j}=...=f(x^k)_{i_k+j}\textrm{ for $0\leq j<m$ and for some $0\leq i_1,...,i_k\leq n-m$}\}.
\end{align*}
For two sequences, this problem has been studied in \cite{CoLaRo}.

Let  $\tilde\Omega=\tilde{\mathcal{A}}^{\mathbb{N}}$ for some alphabet $\tilde{\mathcal{A}}$ and $\tilde{\mathcal{F}}_0^n$ the sigma-algebra generated by the $n$-cylinders in $\tilde\Omega$.

For encoded sequences, to obtain an optimal result, we will need some control on the length of preimage of cylinders:

(HC) $C_n \in \tilde{\mathcal{F}}_0^n$ implies $f^{-1}C_n \in \mathcal{F}_0^{h(n)},$ where $h(n)= o(n^{\gamma})$, for some $\gamma>0$.

\begin{theorem} \label{discrete} Consider $f:\Omega \to \tilde\Omega$ {an encoder} such that $\underline{H}_k(f_*\mathbb{P})>0$. For $\P^k$-almost every $(x^1,\dots,x^k)\in \Omega^k$,
\begin{eqnarray*}
\underset{n\rightarrow+\infty}{\overline\lim}\frac{M_n^f(x^1,\dots,x^k)}{\log n}\leq\frac{k}{(k-1)\underline{H}_k(f_*\P)}.
\end{eqnarray*}
Moreover, if the system $(\Omega, \mathbb{P}, \sigma)$ is $\alpha$-mixing with an exponential decay (or $\psi$-mixing with $\psi(g)=g^{-a}$ for some $a>0$) and (HC) is satisfied, then for $\P^k$-almost every $(x^1,\dots,x^k)\in \Omega^k$,
\begin{eqnarray*}
\underset{n\rightarrow+\infty}{\underline\lim}\frac{M_n^f(x^1,\dots,x^k)}{\log n}\geq\frac{k}{(k-1)\overline{H}_k(f_*\P)}.
\end{eqnarray*}
Therefore, if the generalized R\'enyi entropy exists, then for $\P^k$-almost every $(x^1,\dots,x^k)\in \Omega^k$,
\[ \underset{n\rightarrow+\infty}{\lim}\frac{M_n^f(x^1,\dots,x^k)}{\log n}=\frac{k}{(k-1)H_k(f_*\P)}.\]
\end{theorem}
We emphasize that one cannot obtain this result from Theorem~\ref{seqmat} since in general the pushforward measure $f_*\mathbb{P}$ is not stationary.

\begin{example}[Stochastic scrabble]
We will consider here the stochastic scrabble (defined by \cite{ArMoWa}) where common substring between sequences will be scored depending on the symbols that compose the substring. Thus, shorter substrings could be more significant than longer ones. 

Suppose that each letter $a\in \mathcal{A}$ is associated to a weight $v(a) \in \mathbb{N}^*$. We also denote the score of a string $z_0z_1\dots z_{m-1}$ by $V(z_0\dots z_{m-1}) =\sum_{j=0}^{m-1} v(z_j)$. 

For $x^1,\dots,x^k\in \mathcal{A}^\N$ , we are interested in the $n^{th}$-highest-scoring matching substring:

\begin{align*}
&V_n(x^1,\dots,x^k) \\
&= \max\limits_{0 \leq i_1,...,i_k \leq n-m}\lt\{V(z_0\dots z_{m-1}): \exists \ 1 \leq m \leq n \ \mbox{s.t $\forall \,0\leq j<m$ } \ z_j = x_{i_1+j}^1=\dots=x_{i_k+j}^{k} \rt\}.
\end{align*}

If $(\mathcal{A}^\N,\P,\sigma)$ is an irreducible and aperiodic Markov chain with transition matrix $\lt(p_{ij}\rt)_{i,j}$ on a finite alphabet $\mathcal{A}=\{1,\dots,d\}$ and assuming that $gdc\{v(1), v(2), \ldots, v(d)\}=1$ then for $\P^k$-almost every $(x^1,\dots,x^k)\in (\mathcal{A}^\N)^k$
\begin{equation*}
\lim_{n \to \infty} \frac{V_n(x^1,\dots,x^k)}{\log n}=\frac{k}{-\log q},
\end{equation*}
where $q$ is the largest positive eigenvalue of the matrix $\lt[q^k_{ij}\rt]_{1\leq i,j \leq \sum_{k=1}^d v(k)}$ with
$$
\begin{array}{lll}
q_{i_\ell i_{\ell +1}} = 1 &  \mbox{if} & 1\leq \ell \leq v(i)-1  \ \mbox{and} \ 1\leq i,j \leq d \ ; \\
q_{i_{v(i)} j_{1}} = p_{ij} & \mbox{if}  &  1\leq i,j \leq d \ ; \\
q_{ij}=0 & & \mbox{otherwise}.
\end{array}
$$
As in \cite[Section 2.3]{CoLaRo}, one can prove this result applying Theorem \ref{discrete} with the {encoder} $f$ defined by
\begin{equation*}
\begin{array}{cccl}
  f \  \colon & \  \chi^{\mathbb{N}}&\to& \chi^{\mathbb{N}} \nonumber \\ \label{old}
 & x_0x_1\cdots &\mapsto &\underbrace{x_0x_0\cdots x_0}_{v(x_0)} \underbrace{x_1x_1\cdots x_1}_{v(x_1)} \cdots \underbrace{x_nx_n\cdots x_n}_{v(x_n)} \cdots
 \end{array}
 \end{equation*}
and observing that $M_n^f(x^1,\dots,x^k)=V_n(x^1,\dots,x^k)$. 
\end{example}

\section{Multidimensional piecewise expanding maps}\label{secexp}
In this section, we  apply Theorem~\ref{thprinc} to a family of maps defined by Saussol \cite{Saussol1}: multidimensional piecewise uniformly expanding maps. It was observed in \cite{AFLV11} that these maps generalize Markov maps which also contain one-dimensional piecewise uniformly expanding maps.

Let $N\geq 1$ be an integer. We will work in the Euclidean space $\mathbb{R}^N$. We denote by $B_\epsilon(x)$ the ball with center $x$ and radius $\epsilon$. For a set $E\subset\mathbb{R}^N$, we write 
$$B_\epsilon(E):=\{y\in\mathbb{R}^N: d(y,E)\leq \epsilon\}.$$

\begin{definition}[Multidimensional piecewise expanding systems]
Let $X$ be a compact subset of $\R^N$ with $\overline{X^\circ}=X$ and $T:X \rightarrow X$. The system $(X,T)$ is a multidimensional piecewise expanding system if there exists a family of at most countably many disjoint open sets $U_i \subset X$ and $V_i$ such that $\overline{U_i} \subset V_i$ and maps $T_i:V_i  \rightarrow \R^N$ satisfying for some $0< \alpha \leq1$, for some small enough $ \epsilon_0>0$, and for all $i$:
\begin{enumerate}
\item $T\vert_{U_i}=T_i\vert_{ U_i} $ and $B_{\epsilon_0}(TU_i)\subset T_i(V_i) $;
\item $T_i \in C^1(V_i),T_i \text{ is injective and }  T_i^{-1}\in C^1(T_iV_i) .$ Moreover, there exists a constant $c$, such that for all $\epsilon \leq \epsilon_0, z \in T_iV_i$ and $x,y \in B_{\epsilon}(z)\cap T_iV_i$ we have 
$$|\det D_xT_i^{-1}-\det D_yT_i^{-1}| \leq c \epsilon^{\alpha} |\det D_zT_i^{-1}|;$$
\item $Leb(X\setminus \bigcup_i U_i)=0$;
\item there exists $s=s(T)<1$ such that for all $u,v \in TV_i$ with $d(u,v) \leq \epsilon_0$ we have $d(T_i^{-1}u,T_i^{-1}v)\leq s d(u,v)$;
\item let $G(\epsilon,\epsilon_0):=\sup_x  G(x,\epsilon,\epsilon_0)$ where
$$G(x,\epsilon,\epsilon_0)=\sum_i\frac{Leb(T_i^{-1}B_{\epsilon}(\partial TU_i)\cap B_{(1-s)\epsilon_0}(x))}{m(B_{(1-s)\epsilon_0}(x))},$$
then the number $\eta=\eta(\delta):=s^\alpha +2 \sup_{\epsilon \leq \delta} \frac{G(\epsilon)}{\epsilon^\alpha}\delta^\alpha$ satisfies $\sup_{\delta \leq \epsilon_0} \eta(\delta)<1.$
\end{enumerate}
\end{definition}
We will prove that the multidimensional piecewise expanding systems satisfy the conditions of Theorem~\ref{thprinc}. 
\begin{proposition}
Let $(X,T)$ be a topologically mixing multidimensional piecewise expanding map and $\mu$ be its absolutely continuous invariant probability measure. Then for $\mu^k$-almost every $(x_1,\dots, x_k)\in X^k$,
\[ \underset{n\rightarrow+\infty}{\lim}\frac{\log m_n(x_1,\dots, x_k)}{-\log n}=\frac{k}{(k-1)N}.\]

\end{proposition}
\begin{proof}
First of all, we define the Banach space involved in the mixing conditions. Let $\Gamma \subset X$ be a Borel set. We define the oscillation of $\varphi \in L^1(Leb)$ over $\Gamma$ as

$$osc(\varphi,\Gamma)= \underset{ \Gamma}{ \textrm{ess-sup} } (\varphi)- \underset{ \ \Gamma}{ \textrm{ess-inf} } (\varphi).$$

Now, given real numbers $0<\alpha \leq 1$ and $0<\epsilon_0<1$ consider the following $\alpha$-seminorm
$$|\varphi|_\alpha=\underset{ 0<\epsilon\leq\epsilon_0}{ \sup \  } \epsilon^{ -\alpha} \int_{ X}osc(\varphi,B_\epsilon(x))dx.$$
We observe that $X\ni x \mapsto osc(\varphi,B_\epsilon(x))$ is a measurable function (see \cite{Saussol1})  and 
$$\text{supp} (\text{osc}(\varphi,B_\epsilon(x)))\subset B_\epsilon(\text{supp } \varphi).$$ 

Let $V_{\alpha}$ be the space of $L^1(Leb)-$functions such that $|\varphi|_\alpha<\infty$ endowed with the norm
$$\|\varphi\|_{\alpha}=\|\varphi\|_{L^1(Leb)}+|\varphi|_\alpha.$$
Then
$(V_{\alpha},\|\cdot\|_{\alpha})$ is a Banach space which does not depend on the choice of $\epsilon_0 $  and
 $V_{\alpha} \subset L^\infty $  (see \cite{Saussol1}).

Saussol \cite{Saussol1} proved that for a piecewise expanding map $T:X \longrightarrow X$, where $X\subset \R^N$ is a compact set, there exists an absolutely continuous invariant probability measure $\mu$ with density $h\in V_\alpha$ which enjoys exponential  decay of correlations against $L^1$ observables on $V_{\alpha}$. More precisely, for all $\psi\in V_{\alpha}$, $\phi\in L^1(\mu)$ and $n\in\N^*$,  we have
\[\left|\int_X\psi.\phi\circ T^n\, d\mu-\int_X \psi d\mu\int_X\phi d \mu\right|\leq\|\psi\|_{\alpha}\|\phi\|_1\theta_n,\]
with $\theta_n =a^{n}$ ($0\leq a<1$). This means that the system $(X,T,\mu)$ satisfies the condition (H1) with  $\mathcal{C}=V_\alpha$.

 It remains to show that the system also satisfies the conditions (H2) (with $r_0=\epsilon_0$).
To this end we need to estimate for each $p\in \{1,\cdots,k\},$ the norm $\|\psi_p\|_{\alpha}$, where the functions $\psi_p$ were defined in \eqref{psi1} and \eqref{psi2}. Since $\psi_p\in L^1(Leb)$ we just need to estimate its $\alpha$-seminorm. 

Since
 \begin{align*}
  \textrm{supp} \ osc(\psi_p,B_\epsilon(\cdot)) \subset B_\epsilon(X),
  \end{align*}

we infer that
$$|\psi_p|_\alpha=\underset{ 0<\epsilon\leq\epsilon_0}{ \sup \  } \epsilon^{ -\alpha} \int_{B_\epsilon(X)}osc(\psi_p,B_\epsilon(x))dx.$$ 
 

 For $p=1$ the computation is similar to the one leading to (20) in \cite{BaLiRo} so  we will only treat the case $p\geq 2$.
 
 Let $0<\epsilon\leq\epsilon_0$. First of all, suppose that $r\leq\epsilon$. Since the density $h$ belongs to  $V_{\alpha}\subset L^\infty$, we have $h \leq c$ for some constant $c>0$. Thus, we observe that
 
 \begin{align*}
  osc(\psi_p,B_\epsilon(x))
&\leq \underset{ y\in B(x,\epsilon)\cap X }{ \textrm{ess-sup} } \psi_p(y)\leq\underset{ y\in B(x,\epsilon)\cap X }{ \textrm{ess-sup} } \int_{X^{p-1}}\left[\prod_{j=1}^{p-1}  \mathbbm{1}_{B(x_j,r)}(y)\right]d\mu^{p-1}(x_1,\dots,x_{p-1})\\
  &=\underset{ y\in B(x,\epsilon)\cap X }{ \textrm{ess-sup} } \mu (B(y,r))^{p-1}
  \leq  C_0^{p-1}c^{p-1}\epsilon^{N(p-1)},
  \end{align*}
where $C_0$ denotes the Lebesgue measure of the unit ball in $\mathbb{R}^N$. Now, using the fact that $B_\epsilon(X)\subset B_{\epsilon_0}(X)$ which is a compact set, we conclude that 
\begin{equation*}
 |\psi_p|_\alpha\leq\underset{ 0<\epsilon\leq\epsilon_0}{ \sup \  } \epsilon^{ -\alpha}  C_0^{p-1}c^{p-1}\epsilon^{N(p-1)}Leb( B_\epsilon(X))\nonumber\\
 \leq C \epsilon_0^{ N(p-1)-\alpha}\label{variationpsi1},
 \end{equation*}
where $C=C_0^{p-1}c^{p-1}Leb( X_{\epsilon_0})$.

For simplicity of notation, from now on we write $d\mu^{j-i+1}(i,j)$ instead of $d\mu^{j-i+1}(x_i,\dots,x_{j}).$

Now suppose $r>\epsilon$. Observe that $y\in B(x,\epsilon)$ implies
$ B(x,r-\epsilon)\subset B(y,r)\subset B(x,r+\epsilon)$.
Thus, if $x_p\in B(x,\epsilon)$, we infer that
\begin{eqnarray*}
& & \prod_{l=p+1}^k \mathbbm{1}_{B(x_p,r)}(x_l)\int_{X^{p-1}}\left[\prod_{j=1}^{p-1}\prod_{l=j+1}^{k}  \mathbbm{1}_{B(x_j,r)}(x_l)\right]d\mu^{p-1}(1,{p-1})\\
&&=\prod_{l=p+1}^k \mathbbm{1}_{B(x_p,r)}(x_l)\int_{X^{p-1}}\left[\prod_{j=1}^{p-1}\prod_{\substack{ l=j+1 \\l\neq p}}^{k}  \mathbbm{1}_{B(x_j,r)}(x_l)\right]\left[\prod_{j=1}^{p-1} \mathbbm{1}_{B(x_p,r)}(x_j)\right]d\mu^{p-1}(1,{p-1})\\
&&\leq\prod_{l=p+1}^k \mathbbm{1}_{B(x,r+\epsilon)}(x_l)\int_{X^{p-1}}\left[\prod_{j=1}^{p-1}\prod_{\substack{ l=j+1 \\l\neq p}}^{k}  \mathbbm{1}_{B(x_j,r)}(x_l)\right]\left[\prod_{j=1}^{p-1} \mathbbm{1}_{B(x,r+\epsilon)}(x_j)\right]d\mu^{p-1}(1,{p-1}).
\end{eqnarray*}

Then, we deduce that
\begin{eqnarray}\label{s}
& &\underset{ y\in B(x,\epsilon)\cap X }{ \textrm{ess-sup} } \psi_p(y)\\
&\leq&\prod_{l=p+1}^k \mathbbm{1}_{B(x,r+\epsilon)}(x_l)\int_{X^{p-1}}\prod_{j=1}^{p-1}\prod_{\substack{ l=j+1 \\l\neq p}}^{k}  \mathbbm{1}_{B(x_j,r)}(x_l)\prod_{j=1}^{p-1} \mathbbm{1}_{B(x,r+\epsilon)}(x_j)d\mu^{p-1}(1,{p-1}). \nonumber
\end{eqnarray}
Using similar ideas, one can prove that
\begin{eqnarray}\label{a2}
& &\underset{ y\in B(x,\epsilon)\cap X }{ \textrm{ess-inf} } \psi_p(y)\\
&\geq&\prod_{l=p+1}^k \mathbbm{1}_{B(x,r-\epsilon)}(x_l)\int_{X^{p-1}}\prod_{j=1}^{p-1}\prod_{\substack{ l=j+1 \\l\neq p}}^{k}  \mathbbm{1}_{B(x_j,r)}(x_l)\prod_{j=1}^{p-1} \mathbbm{1}_{B(x,r-\epsilon)}(x_j)d\mu^{p-1}(1,{p-1}). \nonumber
\end{eqnarray}
From \eqref{s} and \eqref{a2} we find that
\begin{eqnarray}
& &\underset{ y\in B(x,\epsilon) \cap X}{ \textrm{ess-sup} } \psi_p({y})- \underset{ \tilde{y}\in B(x,\epsilon) \cap X}{ \textrm{ess-inf} } \psi_p(\tilde{y})\nonumber\\
&\leq&\prod_{l=p+1}^k \mathbbm{1}_{B(x,r+\epsilon)}(x_l)\int_{X^{p-1}}\prod_{j=1}^{p-1}\prod_{\substack{ l=j+1 \\l\neq p}}^{k}  \mathbbm{1}_{B(x_j,r)}(x_l)A_1^{p-1}(x)d\mu^{p-1}(1,p-1)\nonumber\\
& &+A_{p+1}^k(x)\int_{X^{p-1}}\prod_{j=1}^{p-1}\prod_{\substack{ l=j+1 \\l\neq p}}^{k}  \mathbbm{1}_{B(x_j,r)}(x_l)\prod_{j=1}^{p-1} \mathbbm{1}_{B(x,r-\epsilon)}(x_j)d\mu^{p-1}(1,p-1),\nonumber
\end{eqnarray}
where  for $i<j$ we define $A_i^j(x)=\prod_{\theta=i}^j \mathbbm{1}_{B(x,r+\epsilon)}(x_\theta)-\prod_{\theta=i}^j \mathbbm{1}_{B(x,r-\epsilon)}(x_\theta)$.

Therefore
\begin{eqnarray}
\int_{ B_\epsilon(X)}osc(\psi_p,B_\epsilon(x))dx
\leq\int_{ B_\epsilon(X)}\big(\int_{X^{p-1}}A_1^{p-1}(x)d\mu^{p-1}(1,{p-1})
+A_{p+1}^k(x)\big)dx.\label{eqintmult}
\end{eqnarray}
Since $\mu$ is absolutely continuous we observe that
\begin{align}
\int_{X^{p-1}}&A_1^{p-1}(x)d\mu^{p-1}(1,{p-1})
=\mu(B(x,r+\epsilon))^{p-1}-\mu(B(x,r-\epsilon))^{p-1}\nonumber\\
&\leq(p-1)\left(\mu(B(x,r+\epsilon))-\mu(B(x,r-\epsilon)\right)\leq c(p-1)Leb(D(x)),\label{eqdisc1}
\end{align}
where $D(x)=B(x,r+\epsilon)\setminus B(x,r-\epsilon)$. For the second term in \eqref{eqintmult}, we have
\begin{align}
\int_{B_\epsilon(X)}&A_{p+1}^k(x) dx
\leq Leb\left(\bigcap_{l=p+1}^k B(x_l,r+\epsilon)\setminus\bigcap_{l=p+1}^k B(x_l,r-\epsilon)\right)\nonumber\\
&\leq\sum_{l=p+1}^k Leb\left(B(x_l,r+\epsilon)\setminus B(x_l,r-\epsilon)\right)
=\sum_{l=p+1}^kLeb(D(x_l)).\label{eqdisc2}
\end{align}

One can see that for any $y\in\R^N$
\begin{eqnarray}\label{lebdisc}
Leb(D(y))\leq 2C_0\epsilon  \sum_{k=0}^{N-1}\binom{N}{k}
\leq 2^{N+1} C_0\epsilon .
\end{eqnarray}
Finally, from  \eqref{eqintmult} - \eqref{lebdisc} we deduce that
\begin{eqnarray}
 |\psi|_\alpha&\leq&\underset{ 0<\epsilon\leq\epsilon_0}{ \sup \  } \epsilon^{ -\alpha} \left(c(p-1) 2^{N+1} C_0\epsilon Leb( B_\epsilon(X))+(k-p)2^{N+1} C_0\epsilon\right)\leq C_1 \epsilon_0^{ 1-\alpha},\label{variationpsi2}
 \end{eqnarray}
where $C_1=c(p-1) 2^{N+1} C_0 Leb( X_{\epsilon_0})+(k-p)2^{N+1} C_0$.

Thus, from \eqref{variationpsi1} and \eqref{variationpsi2}, we obtain that (H2) is satisfied. Moreover, one can show easily that ${D}_k(\mu)=N$ and the proposition is proved.
\end{proof}



\section{Proof of the symbolic case}\label{sec-discrete}

In this section we prove  the symbolic case (Theorem~\ref{seqmat}).
  We emphasize that even if the proof is based on the ideas of Theorem 7 in \cite{BaLiRo}, the generalisation is not immediate and some extra care is needed. In particular, one need to choose carefully between several different (and equivalent) definitions for $S_n$ (see \eqref{defsnseq} and \eqref{eqsnequi}) so  the proof goes smoothly when using the mixing assumptions. We will focus on these extensions rather than the technical details that are similar to the ideas in \cite{BaLiRo}.

We also observe that next section is dedicated to Theorems \ref{thineq} and  \ref{dsup} whose  proofs follow the lines of the proof of Theorem~\ref{seqmat} but are more complex and technical.
Thus, this section can be seen as a warm-up to Section \ref{sec-proof}.

 
 
We will assume that the system is $\alpha$-mixing with an exponential decay, the $\psi$-mixing case can be easily deduced using the same ideas.

\begin{proof}[Proof of Theorem~\ref{seqmat}-\eqref{eqren1}]

First, for $\eps>0$ and $k_n>0$ let us define

\begin{equation}\label{defsnseq}
S_n(x^1,\dots,x^k)=\sum_{i_1,\dots,i_k =0}^{n-1}\prod_{l=2}^k \mathbbm{1}_{C_{k_n}(\sigma^{i_1} x^1)}(\sigma^{i_l}x^l),
\end{equation}
 and observe that 
\begin{equation}\label{eqsnmnseq}
S_n(x^1,\dots,x^k)\geq 1 \Longleftrightarrow M_n(x^1,\dots,x^k)\geq k_n.
\end{equation}

Next we compute the expectation of $S_n.$
Since $\P$ is a $\sigma$-invariant probability measure we infer that
\begin{align*}
\E(S_n)=&\sum_{i_1,\dots,i_k =0}^{n-1}\int_\Omega\left[\prod_{l=2}^k\int_{\Omega} \mathbbm{1}_{C_{k_n}(\sigma^{i_1} x^1)}(\sigma^{i_l}x^l)d\P(x^l)\right]d\P(x_1)\\
=&\sum_{i_1,\dots,i_k =0}^{n-1}\int_\Omega\P\left({C_{k_n}(\sigma^{i_1} x^1)}\right)^{k-1}d\P(x_1)=n^k\int_\Omega\P\left({C_{k_n}(x^1)}\right)^{k-1}d\P(x_1).
\end{align*}

Using the partition of \emph{$k_n$-cylinders}, we infer that
\begin{eqnarray}
\E(S_n)&=&n^k \sum_{C_{k_n}}  \int_{C_{k_n}} \P\left(C_{k_n}\cap C_{k_n}( x^1)\right)^{k-1}d\P(x^1)=n^k \sum_{C_{k_n}}  \P\left(C_{k_n}\right)^{k}.\label{esperanca}
\end{eqnarray}
Now we are ready to prove \eqref{eqren1}.

Define ${k}_n=\frac{1}{(k-1)\underline{H}_k-\eps}(k\log n+\log\log n).$ From \eqref{eqsnmnseq}, \eqref{esperanca} and Markov's inequality, we find that
\begin{eqnarray*}
\P^k\left(M_n(x^1,\dots,x^k)\geq k_n\right)
\leq \E(S_n)=n^k \sum_{C_{k_n}}  \P\left(C_{k_n}\right)^{k}.
\end{eqnarray*}

By the definition of the lower entropy and the definition of $k_n$, for $n$ large enough, we have
\[\P^k\left(M_n(x^1,\dots,x^k)\geq k_n\right)\leq \frac{1}{\log n}.\]
Finally, choosing a subsequence $n_\ell=\lceil e^{\ell^2}\rceil$, we know that 
\[\P^k\left(M_n(x^1,\dots,x^k)\geq k_{n_\ell}\right)\leq \frac{1}{\log n_\ell}\leq \frac{1}{\ell^2}.\]
Thus $\sum_\ell \P^k\left(M_n(x^1,\dots,x^k)\geq k_{n_\ell}\right)<+\infty$. By Borel-Cantelli Lemma, we know that
\[\frac{\log M_{n_\ell}(x^1,\dots,x^k)}{\log n_\ell}\leq \frac{1}{(k-1)\underline{H}_k-\eps}\left(k+\frac{\log\log n_\ell}{\log n_\ell}\right),\]
for $\P^k$-almost every $(x^1,\dots,x^k)\in\Omega^k$ and $\ell$ large enough.
Since $\eps>0$ can be chosen arbitrarily small we obtain that
\begin{equation*}\underset{\ell\rightarrow+\infty}{\overline\lim}\frac{\log M_{n_\ell}(x^1,\dots,x^k)}{\log n_\ell}\leq \frac{k}{(k-1)\underline{H}_k}.\end{equation*}
Observing that  $(n_\ell)_\ell$ is increasing, $(\log M_n)_n$ is decreasing and $\underset{\ell\rightarrow+\infty}\lim\frac{\log n_\ell}{\log n_{\ell+1}}=1$, we infer that the last inequality holds if we replace $n_l$ by $n$,
and \eqref{eqren1} is proved.
\end{proof}
\begin{proof}[Proof of Theorem~\ref{seqmat}-\eqref{eqren2}]

Let $b<0$ to be choosen later. To prove $\eqref{eqren2}$, we set
\[k_n=\frac{1}{(k-1)\overline{H}_k+\eps}(k\log n+b\log\log n).  \]

From \eqref{eqsnmnseq} and Chebychev's inequality we infer that
\begin{equation}\label{eqvarshseq}
P^k\left(M_n(x^1,\dots,x^k)< k_n\right)= \P^k \left(S_n(x^1,\dots,x^k)=0\right)\leq\frac{\var(S_n)}{\E(S_n)^2}.
\end{equation}
In order to bound $\frac{\var(S_n)}{\E(S_n)^2}=\frac{\E(S_n^2)-\E(S_n)^2}{\E(S_n)^2}$ we need to analyse the term


\begin{eqnarray}
 \E(S_n^2)=\sum_{\substack{i_1,\dots,i_k =0,\dots,n-1\\ i'_1,\dots,i'_k =0,\dots,n-1}}\int_{\Omega^k}\prod_{l=2}^k \mathbbm{1}_{C_{k_n}(\sigma^{i_1} x^1)}(\sigma^{i_l}x^l)\mathbbm{1}_{C_{k_n}(\sigma^{i'_1} x^1)}(\sigma^{i'_l}x^l)d\P^k(x^1,\dots,x^k).\label{eqdefvar}
\end{eqnarray}
 We will split this sum in two cases depending on the relative position of $i_l$ and $i'_l$.

Let $g=g(n)=\log(n^{2k+1})$. First of all, we observe that if $|i_l-i'_l |>g+k_n$ then the $\alpha-$mixing condition gives that

\begin{align}\label{far}
\int_{\Omega} \mathbbm{1}_{C_{k_n}(\sigma^{i_1} x^1)}(\sigma^{i_l}x^l)\mathbbm{1}_{C_{k_n}(\sigma^{i'_1} x^1)}&(\sigma^{i'_l}x^l)d\P(x^l)\nonumber\\
&\leq\alpha(g)+\P\left(C_{k_n}(\sigma^{i_1} x^1)\right)\P\left( C_{k_n}(\sigma^{i'_1} x^1)\right).
\end{align}

If otherwise  $|i_l-i'_l |\leq g+k_n$ then H\"older's inequality infers that
\begin{eqnarray}\label{ineqholderseq}
\int_{\Omega} \mathbbm{1}_{C_{k_n}(\sigma^{i_1} x^1)}(\sigma^{i_l}x^l)\mathbbm{1}_{C_{k_n}(\sigma^{i'_1} x^1)}(\sigma^{i'_l}x^l)d\P(x^l)
\leq\P\left(C_{k_n}(\sigma^{i_1} x^1)\right)^{1/2}\P\left( C_{k_n}(\sigma^{i'_1} x^1)\right)^{1/2}.
\end{eqnarray}

Now suppose that $i_l-i'_l >g+k_n$ for every $l\in\{1,\dots,k\}$ (the case $i'_l-i_l >g+k_n$ can be treated identically). From \eqref{far} we find  that
\begin{eqnarray}
& &\int_{\Omega^k}\prod_{l=2}^k \mathbbm{1}_{C_{k_n}(\sigma^{i_1} x^1)}(\sigma^{i_l}x^l)\mathbbm{1}_{C_{k_n}(\sigma^{i'_1} x^1)}(\sigma^{i'_l}x^l)d\P^k(x^1,\dots,x^k)\label{ineqallmixseq}\\
&&=\int_\Omega\left[\prod_{l=2}^k\int_{\Omega} \mathbbm{1}_{C_{k_n}(\sigma^{i_1} x^1)}(\sigma^{i_l}x^l)\mathbbm{1}_{C_{k_n}(\sigma^{i'_1} x^1)}(\sigma^{i'_l}x^l)d\P(x^l)\right]d\P(x_1)\nonumber\\
&&\leq\int_\Omega \left(\alpha(g)+\P\left(C_{k_n}(\sigma^{i_1} x^1)\right)\P\left( C_{k_n}(\sigma^{i'_1} x^1)\right)\right)^{k-1}d\P(x^1)\nonumber\\
&&\leq(2^{k-1}-1)\alpha(g)+\int_\Omega \P\left(C_{k_n}(\sigma^{i_1} x^1)\right)^{k-1}\P\left( C_{k_n}(\sigma^{i'_1} x^1)\right)^{k-1}d\P(x^1).\nonumber
\end{eqnarray}

To conclude the first case we use the partition $\{C_{k_n}\cap \sigma^{-(i_1 -i'_1)}C_{k_n}' \}_{C_{k_n},C'_{k_n}}$  of $\Omega$ to infer that
\begin{eqnarray}
& &\int_\Omega \P\left(C_{k_n}(\sigma^{i_1} x^1)\right)^{k-1}\P\left( C_{k_n}(\sigma^{i'_1} x^1)\right)^{k-1}d\P(x^1)\label{ineqallmixseq1}\\
&&=\sum_{C_{k_n},C'_{k_n}}\int_{C_{k_n}\cap \sigma^{-(i_1 -i'_1)}C_{k_n}' } \P\left(C_{k_n}(\sigma^{i_1-i'_1} x^1)\right)^{k-1}\P\left( C_{k_n}(x^1)\right)^{k-1}d\P(x^1)\nonumber\\
&&=\sum_{C_{k_n},C'_{k_n}}\P(C_{k_n}\cap \sigma^{-(i_1 -i'_1)}C_{k_n}' ) \P\left(C_{k_n}\right)^{k-1}\P\left( C'_{k_n}\right)^{k-1}\leq \alpha(g)+\left(\sum_{C_{k_n}} \P\left(C_{k_n}\right)^{k}\right)^2.\nonumber
\end{eqnarray}

Next, for $p\in\{1,\dots,k\}$, assume that we have $p$ pairs of close indices and $k-p$ pairs of distant indices. We will firstly treat the case where $|i_1-i'_1 |\leq g+k_n$. Without loss of generality we can assume $|i_2-i'_2 |\leq g+k_n$,\dots,$|i_p-i'_p |\leq g+k_n$,
 $|i_{p+1}-i'_{p+1} |> g+k_n$,\dots,\\$|i_{k}-i'_{k} |> g+k_n$. From \eqref{far} and \eqref{ineqholderseq} we deduce that

\begin{eqnarray}
& &\int_{\Omega^k}\prod_{l=2}^k \mathbbm{1}_{C_{k_n}(\sigma^{i_1} x^1)}(\sigma^{i_l}x^l)\mathbbm{1}_{C_{k_n}(\sigma^{i'_1} x^1)}(\sigma^{i'_l}x^l)d\P^k(x^1,\dots,x^k)\label{ineqmixhol}\\
&&\leq\int_\Omega\left(\alpha(g)+\P\left(C_{k_n}(\sigma^{i_1} x^1)\right)\P\left( C_{k_n}(\sigma^{i'_1} x^1)\right)\right)^{k-p}\times\nonumber\\
&&\hspace{2cm}\times\left(\P\left(C_{k_n}(\sigma^{i_1} x^1)\right)^{1/2}\P\left( C_{k_n}(\sigma^{i'_1} x^1)\right)^{1/2}\right)^{p-1}d\P(x^1)\nonumber\\
&&\leq(2^{k-p}-1)\alpha(g)+\int_\Omega\P\left(C_{k_n}(\sigma^{i_1} x^1)\right)^{k-(p+1)/2}\P\left( C_{k_n}(\sigma^{i'_1} x^1)\right)^{k-(p+1)/2}d\P(x^1).\nonumber
\end{eqnarray}
Using H\"older's inequality and the invariance of $\P$, we obtain 
\begin{eqnarray}
& &\int_\Omega\P\left(C_{k_n}(\sigma^{i_1} x^1)\right)^{k-(p+1)/2}\P\left( C_{k_n}(\sigma^{i'_1} x^1)\right)^{k-(p+1)/2}d\P(x^1)\nonumber\\
&&\leq\int_\Omega\P\left(C_{k_n}(x^1)\right)^{2k-(p+1)}d\P(x^1)
=\sum_{C_{k_n}} \P\left(C_{k_n}\right)^{2k-p}
\leq\left(\sum_{C_{k_n}} \P\left(C_{k_n}\right)^{k}\right)^{(2k-p)/k},\label{ineqsubba}
\end{eqnarray}
where the last inequality came from the fact that $x\mapsto x^{k/(p+k)}$ is a countably subadditive function.

If $|i_1-i'_1 |> g+k_n$ then since we have $p\geq 1$ pairs of close indices, 
 there exists at least one $j\in\{2,\dots,k\}$ such that $|i_j-i'_j |\leq g+k_n$. In this case, the estimations \eqref{ineqmixhol}  and \eqref{ineqsubba} could be done similarly using the following equivalent definition of $S_n$
\begin{equation}\label{eqsnequi}
S_n(x^1,\dots,x^k)=\sum_{i_1,\dots,i_k =0}^{n-1}\prod_{\substack{l=1 \\ l\neq j}}^k \mathbbm{1}_{C_{k_n}(\sigma^{i_j} x^j)}(\sigma^{i_l}x^l).
\end{equation}

Now, gathering the estimates \eqref{eqdefvar} and \eqref{ineqallmixseq}- \eqref{ineqsubba} we  conclude that
\begin{eqnarray}
\var(S_n)
&\leq& n^{2k}3^{k}\alpha(g)+\sum_{p=1}^{k}\left[\binom{k}{p}n^{2k-p}(g+k_n)^p\left(\sum_{C_{k_n}} \P\left(C_{k_n}\right)^{k}\right)^{(2k-p)/k}\right]\nonumber\\
&=&n^{2k}3^{k}\alpha(g)+\sum_{p=1}^{k}\left[\binom{k}{p}(g+k_n)^p\left(\E(S_n)\right)^{(2k-p)/k}\right].
\label{ineqvarfinseq}
\end{eqnarray}
Thus, \eqref{ineqvarfinseq} together with \eqref{eqvarshseq} gives us
\begin{equation*}
\P^k\left(M_n(x^1,\dots,x^k)< k_n\right)\leq\frac{n^{2k}3^{k}\alpha(g)}{\E(S_n)^2}+\sum_{p=1}^{k}\frac{\binom{k}{p}(g+k_n)^p}{\left(\E(S_n)\right)^{p/k}}.
\end{equation*}
By the definitions of $k_n$ and \eqref{esperanca}, we observe that for $n$ large enough we have
$\E(S_n)\geq(\log n)^{-b},$
and since  $g=\log\left(n^{2k+1}\right)$, we infer that
\[\frac{n^{2k}3^{k}\alpha(g)}{\E(S_n)^2}=\mathcal{O}\left(\frac{1}{\log n}\right).\]
We can choose $b\ll-1$ so that
\[\P^k\left(M_n(x^1,\dots,x^k)< k_n\right)=\mathcal{O}\left(\frac{1}{\log n}\right).\]
To conclude the proof it suffices to take a subsequence $n_\ell$ and use Borel-Cantelli Lemma as in the proof of \eqref{eqren1}.

\end{proof}

\begin{proof}[Proof of Theorem~\ref{discrete}]
The proof follows the line of the proof of Theorem~\ref{seqmat}, replacing $S_n$ by
\begin{equation*}
S_n^f(x^1,\dots,x^k)=\sum_{i_1,\dots,i_k =0}^{n-1}\prod_{l=2}^k \mathbbm{1}_{f^{-1}C_{k_n}(f(\sigma^{i_1} x^1))}(\sigma^{i_l}x^l).
\end{equation*}
Moreover, since $f$ can modify the length of cylinders, while using the mixing property, one need to use assumption (HC) and $\alpha(g)$ must be replaced by $\alpha(g+k_n-h_n)$.
\end{proof}

\section{Proofs of the main results}\label{sec-proof}
In this section we adapt  the proof of Theorem~\ref{seqmat} for multiple orbits (Theorems \ref{thineq} and \ref{dsup}).
In order to do that, one must replace $M_n$ by $-\log m_n$ and the cylinders $C_k(x)$ by balls $B(x, e^{-k})$. However, one major drawback is that for cylinders we have that $x\in C_n(y)$ implies that $C_n(y)=C_n(x)$ but, when working with balls, $x\in B(y,r)$ does not implies that $B(y,r)=B(x,r)$. 
This simple fact prohibits us to define $S_n$ as in the previous section, in particular in view of \eqref{eqsnequi}. To overcome this problem we will need to define $S_n$ as
\begin{equation}\label{Sn}S_n(x_1,\dots,x_k)=\sum_{i_1,\dots,i_k =0}^{n-1}\prod_{j=1}^{k-1} \prod_{l=j+1}^k \mathbbm{1}_{B(T^{i_j}x_j,r_n)}(T^{i_l}x_l)\end{equation}
which will complexify our proofs. In particular, we will need to use the following lemma in the proof of both theorems.
\begin{lemma}\label{lemmagendim}
\begin{equation*}(k-1)\underline{D}_k(\mu)=\underset{r\rightarrow0}{\underline\lim}\frac{\log\int_{X^k}\prod_{j=1}^{k-1} \prod_{l=j+1}^k \mathbbm{1}_{B(x_j,r)}(x_l) d\mu^k(1,k)}{\log r}
\end{equation*}
and
\begin{equation*}(k-1)\overline{D}_k(\mu)=\underset{r\rightarrow0}{\overline\lim}\frac{\log\int_{X^k}\prod_{j=1}^{k-1} \prod_{l=j+1}^k \mathbbm{1}_{B(x_j,r)}(x_l) d\mu^k(1,k)}{\log r}
\end{equation*}
\end{lemma}
\begin{proof}
First of all, one can observe that for every $(x_1,\dots,x_k)\in X^k$
\[
\prod_{j=1}^{k-1} \prod_{l=j+1}^k \mathbbm{1}_{B(x_j,r)}(x_l) \leq  \prod_{l=2}^k \mathbbm{1}_{B(x_1,r)}(x_l). \]
Thus, 
\begin{eqnarray}
\int_{X^k}\prod_{j=1}^{k-1} \prod_{l=j+1}^k \mathbbm{1}_{B(x_j,r)}(x_l)d\mu^k(1,k) 
\leq\int_X\mu\left(B\left(x,r\right)\right)^{k-1}d\mu(x) \label{ineqd1}.
\end{eqnarray}
Moreover, one can observe that if $\{x_i,x_j \}\subset B(x_1,r/2)$  then $x_i \in B(x_j,r)$. Therefore
\[
\prod_{l=2}^k \mathbbm{1}_{B(x_1,r/2)}(x_l) \leq\prod_{j=1}^{k-1} \prod_{l=j+1}^k \mathbbm{1}_{B(x_j,r)}(x_l) \]
for every $(x_1,\dots,x_k)\in X^k$, which implies that
\begin{equation}\label{ineqd2}
\int_X\mu\left(B\left(x,\frac{r}{2}\right)\right)^{k-1}d\mu(x)\leq \int_{X^k}\prod_{j=1}^{k-1} \prod_{l=j+1}^k \mathbbm{1}_{B(x_j,r)}(x_l)d\mu^k(1,k).
\end{equation}
Using \eqref{ineqd1} and \eqref{ineqd2} and the fact that $\underset{r\rightarrow 0}{\lim}\frac{\log(r/2)}{\log r}=1$ we get the result.
\end{proof}

\begin{proof}[Proof of Theorem \ref{thineq}]
As in the proof of Theorem~\ref{seqmat}-\eqref{eqren2} it suffices to show that
\[\mu^k \left(m_n(x_1,\dots,x_k)< r_n\right)=\mathcal{O}\left(\frac{1}{\log n}\right).\]

For $\eps>0$, let us define 
\[k_n=\frac{1}{(k-1)\underline{D}_k(\mu)-\eps}(k\log n+\log\log n) \quad \text{and} \quad r_n=e^{-k_n}.\]
Defining $S_n(x_1,\dots,x_k)$ as in \eqref{Sn}, it is easy to see that
 for every $(x_1,\dots,x_k)\in X^k$
\begin{equation}\label{eqMnSn}
m_n(x_1,\dots,x_k)< r_n\Longleftrightarrow S_n(x_1,\dots,x_k)\geq 1,\end{equation}
where $m_n$ was defined in \eqref{mn}. 
Then, from \eqref{eqMnSn} and Markov's inequality, we deduce that
\begin{eqnarray*}
\mu^k \left(m_n(x_1,\dots,x_k)< r_n\right)&\leq& \E(S_n)=\\
&=&\sum_{i_1,\dots,i_k =0}^{n-1}\int_{X^k}\prod_{j=1}^{k-1} \prod_{l=j+1}^k \mathbbm{1}_{B(x_j,r_n)}(x_l)d\mu^k(1,k)\\
&\leq & n^k\int_{X^k}\prod_{j=1}^{k-1} \prod_{l=j+1}^k \mathbbm{1}_{B(x_j,r_n)}(x_l)d\mu^k(1,k),
\end{eqnarray*}
since $\mu$ is invariant.

By Lemma \ref{lemmagendim} and the definition of $k_n$, for $n$ large enough, we infer that
\[\mu^k \left(m_n(x_1,\dots,x_k)< r_n\right)\leq n^kr_n^{(k-1)\underline{D}_k(\mu)-\eps}=\frac{1}{\log n},\]
and this is the desired conclusion.
\end{proof}

Before proving Theorem \ref{dsup} we state a few facts in order to 
 simplify the calculations.
At first let us recall the notion of $(\lambda,r)$-grid partition.
\begin{definition}
Let $0<\lambda<1$ and $r>0$. A partition $\{Q_i\}_{i=1}^{\infty}$ of $X$ is called a $(\lambda,r)$-grid partition if there exists a sequence $\{y_i\}_{i=1}^{\infty}$ such that for any $i\in\N$
\[B(y_i,\lambda r)\subset Q_i\subset B(y_i, r).\]
\end{definition}
The following technical lemma will be used during the proof. One can observe that in the symbolic case, this lemma corresponds to \eqref{ineqsubba}. Moreover, this lemma is a generalization of Lemma 14 in \cite{BaLiRo}.
\begin{lemma}\label{lema1}
Let $p\in \{1,\dots, k-1\}$. Under the hypotheses of Theorem \ref{dsup}, there exists a constant $K>0$ such that for $n$ large enough
\begin{align*}
\int_{X^{k-p}}\prod_{j=p+1}^{k-1} &\prod_{l=j+1}^k \mathbbm{1}_{B(x_j,r_n)}(x_l)\left(\int_{X^p}\prod_{j=1}^{p}\prod_{l=j+1}^{k}  \mathbbm{1}_{B(x_j,r_n)}(x_l)d\mu^p(1,p)\right)^2d\mu^{k-p}({p+1},k)\\
&\leq K\left(\int_{X^k}\prod_{j=1}^{k-1} \prod_{l=j+1}^k \mathbbm{1}_{B(x_j,r_n)}(x_l)d\mu^k(1,k)\right)^{(p+k)/k}=K\left(\frac{\E(S_n)}{n^k}\right)^{\frac{p+k}{k}},
\end{align*}
where $d\mu^{j-i+1}(i,j)$ denotes $d\mu^{j-i+1}(x_i,\dots,x_{j}),\ \text{for}\ i<j.$
\end{lemma}
\begin{proof}

By Proposition 2.1 in \cite{GY}, there exist $0<\lambda<\frac12$ and $R>0$ such that for any $0<r<R$ there exists a $(\lambda,r)$-grid partition.

Given $r_0$ as in definition~\ref{deftight} let us choose $n$ large enough so that $r_n<\min\{R,r_0/2\}$. Let $\{Q_i\}_{i=1}^{\infty}$ be a $(\lambda,\frac{r_n}{2})$-grid partition and $\{y_i\}_{i=1}^\infty$ be such that $$B\left(y_i,\lambda \frac{r_n}{2}\right)\subset Q_i
\subset B\left(y_i, \frac{r_n}{2}\right).$$

Using this partition we infer that
\begin{eqnarray}
& &\int_{X^{k-p}}\prod_{j=p+1}^{k-1} \prod_{l=j+1}^k \mathbbm{1}_{B(x_j,r_n)}(x_l)\left(\int_{X^p}\prod_{j=1}^{p}\prod_{l=j+1}^{k}  \mathbbm{1}_{B(x_j,r_n)}(x_l)d\mu^p(1,p)\right)^2d\mu^{k-p}({p+1},k)\nonumber\\
&&\leq \int_{X^{k-p}} \prod_{l=p+2}^k \mathbbm{1}_{B(x_{p+1},r_n)}(x_l)\left(\int_{X^p}\prod_{j=1}^{p}  \mathbbm{1}_{B(x_{j},r_n)}(x_{p+1})d\mu^p(1,p)\right)^2d\mu^{k-p}(p+1,k)\nonumber\\
&&= \int _X \mu\left(B(x_{p+1},r_n)\right)^{p+k-1}d\mu(x_{p+1})=\sum_{i}\int_{Q_i} \mu\left(B(x_{p+1},r_n)\right)^{p+k-1}d\mu(x_{p+1}).\label{ineqlemdim1}
\end{eqnarray}


Now, for $i$ fixed, there exist $k_i$ elements $\{Q_{i,j}\}_{j=1}^{k_i}$ of the partition $\{Q_k\}_{k=1}^{\infty}$ such that $Q_{i,j} \cap B(y_i,2r_n) \neq \emptyset$ for $j=1,...,k_i.$
 Since the space is tight, there exists a constant $K_0$ depending only on $N_0$ such that  $k_i\leq K_0$ (see e.g. the proof of Theorem 4.1 in \cite{GY}). Defining $Q_{i,j}=\emptyset$ for $k_i<j\leq K_0$ we infer that
\begin{equation}\label{a1}\bigcup_{x_{p+1} \in Q_i}B(x_{p+1},r_n)\subset B(y_i,2r_n) \subset\bigcup_{j=1}^{K_0}Q_{i,j}.\end{equation}

From  \eqref{a1} we know that
\begin{eqnarray*}
& &\sum_{i}\int_{Q_i} \mu\left(B(x_{p+1},r_n)\right)^{p+k-1}d\mu(x_{p+1})
\leq\sum_{i}\int_{Q_i}\left(\sum_{j=1}^{K_0}\mu\left(Q_{i,j}\right)\right)^{p+k-1}d\mu(x_{p+1})\\
&=&\sum_{i}\mu(Q_i)\left(\sum_{j=1}^{K_0}\mu\left(Q_{i,j}\right)\right)^{p+k-1}
\leq\sum_{i}\left(\sum_{j=1}^{K_0}\mu\left(Q_{i,j}\right)\right)^{p+k}\leq K_0^{p+k-1}\sum_{i}\sum_{j=1}^{K_0}\mu\left(Q_{i,j}\right)^{p+k},
\end{eqnarray*}
where the last inequality is deduced from Jensen's inequality.
Now, since the elements $Q_{i,j}$ cannot participate in more than $K_0$ different sums (one can see the arguments leading to (12) in \cite{GY}) and since $x\mapsto x^{k/(p+k)}$ is a countably subadditive function, we infer that
\begin{eqnarray}
& &\sum_{i}\int_{Q_i} \mu\left(B(x_{p+1},r_n)\right)^{p+k-1}d\mu(x_{p+1})
\leq K_0^{p+k}\sum_{i}\mu\left(Q_{i}\right)^{p+k}\nonumber\\
&&\leq K_0^{p+k}\left(\sum_{i}\mu\left(Q_{i})\right)^k\right)^{(p+k)/k}=K_0^{p+k}\left(\sum_{i}\int_{X^k}\prod_{l=1}^k\mathbbm{1}_{Q_i}(x_l) d\mu^k(1,k)\right)^{(p+k)/k}.\label{ineqlemdim2}
\end{eqnarray}
Note that for any $y\in Q_i$, we have $Q_i\subset B(y,r_n)$. Thus, if $\{x_1,\cdots,x_k \}\subset Q_i$, then we have $x_l\in B(x_j, r_n)$ for any $j,l=1,\dots,k$ and we conclude that
\begin{equation}\label{ineqlemdim3}\sum_{i}\int_{X^k}\prod_{l=1}^k\mathbbm{1}_{Q_i}(x_l) d\mu^k(1,k)\leq\int_{X^k}\prod_{j=1}^{k-1} \prod_{l=j+1}^k \mathbbm{1}_{B(x_j,r_n)}(x_l)d\mu^k(1,k).
\end{equation}
Finally, \eqref{ineqlemdim1}, \eqref{ineqlemdim2} and \eqref{ineqlemdim3} give us
\begin{multline}\nonumber
\int_{X^{k-p}}\prod_{j=p+1}^{k-1} \prod_{l=j+1}^k \mathbbm{1}_{B(x_j,r_n)}(x_l)\left(\int_{X^p}\prod_{j=1}^{p}\prod_{l=j+1}^{k}  \mathbbm{1}_{B(x_j,r_n)}(x_l)d\mu^p(1,p)\right)^2d\mu^{k-p}(p+1,k)\\
\leq K_0^{p+k}\left(\int_{X^k}\prod_{j=1}^{k-1} \prod_{l=j+1}^k \mathbbm{1}_{B(x_j,r_n)}(x_l)d\mu^k(1,k)\right)^{(p+k)/k}.
\end{multline}
 and the result follows with $K=K_0^{p+k}$. 
\end{proof}


We are now ready to prove Theorem \ref{dsup}.
\begin{proof}[Proof of Theorem \ref{dsup}]
Without loss of generality, we will assume in the proof that $\theta_n=e^{-n}$.

For $\eps>0$, let us define 
\[k_n=\frac{1}{(k-1)\overline{D}_k(\mu)+\eps}(k\log n+b\log\log n)  \quad \text{and} \quad r_n=e^{-k_n}.\]
Using the same notation as in the proof of Theorem \ref{thineq}, we recall that
\begin{equation}\label{eqexp}\E(S_n)=n^k\int_{X^k}\prod_{j=1}^{k-1} \prod_{l=j+1}^k \mathbbm{1}_{B(x_j,r_n)}(x_l)d\mu^k(1,k).\end{equation}
To simplify our equations, from now on, we will denote by $B(x_j)$ the set $B(x_j,r_n).$

Using \eqref{eqMnSn} and Chebyshev's inequality, we obtain
\begin{equation}\label{eqmnvarsn}
\mu^k \left(m_n(x_1,\dots,x_k)\geq r_n\right)\leq\mu^k \left(S_n(x_1,\dots,x_k)=0\right)\leq
\frac{\var(S_n)}{\E(S_n)^2}.
\end{equation}
Thus, we need to control the variance of $S_n$. First of all, we have
\begin{equation*}
\var(S_n)
=\sum_{\substack{i_1,\dots,i_k =0,\dots,n-1\\ i'_1,\dots,i'_k =0,\dots,n-1}}\int_{X^k}\prod_{j=1}^{k-1} \prod_{l=j+1}^k \mathbbm{1}_{B(T^{i_j}x_j)}(T^{i_l}x_l)  \mathbbm{1}_{B(T^{i'_j}x_j)}(T^{i'_l}x_l)d\mu^k(1,k)-\E(S_n)^2.
\end{equation*}
Let $g=g(n)=\log(n^{\gamma})$ where $\gamma>0$ will be defined later. 

We will split the last sum depending on the relative position of $i_l$ and $i'_l$. Without loss of generality we can always suppose
$i_l>i'_l$.

We first consider the case $i_1-i'_1> g,\dots, i_k-i'_k> g.$

Since $i_1-i'_1>g$ then by (H1) and (H2),
\begin{eqnarray}
& &\int_{X^k}\prod_{j=1}^{k-1} \prod_{l=j+1}^k \mathbbm{1}_{B(T^{i_j}x_j)}(T^{i_l}x_l)  \mathbbm{1}_{B(T^{i'_j}x_j)}(T^{i'_l}x_l)d\mu^k(1,k)\label{eqmix1}\\
&=&\int_{X^{k-1}}\prod_{j=2}^{k-1} \prod_{l=j+1}^k \mathbbm{1}_{B(T^{i_j}x_j)}(T^{i_l}x_l)  \mathbbm{1}_{B(T^{i'_j}x_j)}(T^{i'_l}x_l)
\times\nonumber\\
&&\hspace{0.5cm}\times\left[\int_X \prod_{l=2}^k \mathbbm{1}_{B(T^{i_1-i'_1}x_1)}(T^{i_l}x_l)  \mathbbm{1}_{B(x_1)}(T^{i'_l}x_l)d\mu(x_1) \right]d\mu^{k-1}(2,k)\nonumber\\
&&\leq\int_{X^{k-1}}\prod_{j=2}^{k-1} \prod_{l=j+1}^k \mathbbm{1}_{B(T^{i_j}x_j)}(T^{i_l}x_l)  \mathbbm{1}_{B(T^{i'_j}x_j)}(T^{i'_l}x_l)\nonumber\\
& &\hspace{0.5cm} \times\left[\int_X \prod_{l=2}^k \mathbbm{1}_{B(x_1)}(T^{i_l}x_l)d\mu(x_1)\right] \left[\int_X \prod_{l=2}^k \mathbbm{1}_{B(x_1)}(T^{i'_l}x_l)d\mu(x_1) \right]d\mu^{k-1}(2,k)+c^2r_n^{-2\xi}\theta_g\nonumber\\
&=:&I+c^2r_n^{-2\xi}\theta_g.\nonumber
\end{eqnarray}

Now we use that $i_2-i'_2>g$ and the same ideas to find that
\begin{eqnarray}
I&=& \int_{X^{k-2}}\prod_{j=3}^{k-1} \prod_{l=j+1}^k \mathbbm{1}_{B(T^{i_j}x_j)}(T^{i_l}x_l)  \mathbbm{1}_{B(T^{i'_j}x_j)}(T^{i'_l}x_l)\nonumber\\
& & \times\int_X\left[\prod_{l=3}^k \mathbbm{1}_{B(T^{i_2-i'_2}x_2)}(T^{i_l}x_l)\int_X \mathbbm{1}_{B(x_1)}(T^{i_2-i'_2}x_2)\prod_{l=3}^k \mathbbm{1}_{B(x_1)}(T^{i_l}x_l)d\mu(x_1)\right] \label{z}\\
& &\qquad\times\left[\prod_{l=3}^k \mathbbm{1}_{B(x_2)}(T^{i'_l}x_l)\int_X \mathbbm{1}_{B(x_1)}(x_2)\prod_{l=3}^k \mathbbm{1}_{B(x_1)}(T^{i'_l}x_l)d\mu(x_1) \right]d\mu(x_2)d\mu^{k-2}(3,k)\nonumber\\
&&\leq \int_{X^{k-2}}\prod_{j=3}^{k-1} \prod_{l=j+1}^k \mathbbm{1}_{B(T^{i_j}x_j)}(T^{i_l}x_l)  \mathbbm{1}_{B(T^{i'_j}x_j)}(T^{i'_l}x_l)\nonumber\\
& & \times\int_{X^2}\mathbbm{1}_{B(x_1)}(x_2)\prod_{l=3}^k  \mathbbm{1}_{B(x_1)}(T^{i_l}x_l)\mathbbm{1}_{B(x_2)}(T^{i_l}x_l)d\mu^2(x_1,x_2) \nonumber\\
& & \times\int_{X^2}\mathbbm{1}_{B(x_1)}(x_2)\prod_{l=3}^k  \mathbbm{1}_{B(x_1)}(T^{i'_l}x_l)\mathbbm{1}_{B(x_2)}(T^{i'_l}x_l)d\mu^2(x_1,x_2) d\mu^{k-2}((3,k))+c^2r_n^{-2\xi}\theta_g.\nonumber
\end{eqnarray}

Applying this argument again we will have on the $p$-th step  ($i_p-i'_p>g$) 
\begin{eqnarray}
&&\int_{X^k}\prod_{j=1}^{k-1} \prod_{l=j+1}^k \mathbbm{1}_{B(T^{i_j}x_j)}(T^{i_l}x_l)  \mathbbm{1}_{B(T^{i'_j}x_j)}(T^{i'_l}x_l)d\mu^k(1,k)\label{eqmixp}\\
&& \leq\int_{X^{k-p}}\prod_{j=p+1}^{k-1} \prod_{l=j+1}^k \mathbbm{1}_{B(T^{i_j}x_j)}(T^{i_l}x_l)  \mathbbm{1}_{B(T^{i'_j}x_j)}(T^{i'_l}x_l)\nonumber\\
&&\hspace{0.5cm}  \times\int_{X^p}\left[\prod_{j=1}^{p-1}\prod_{l=j+1}^{p}  \mathbbm{1}_{B(x_j)}(x_l)\right]\left[\prod_{j=1}^p\prod_{l=p+1}^k  \mathbbm{1}_{B(x_j)}(T^{i_l}x_l)\right]d\mu^p(1,p) \nonumber\\
&  &\hspace{0.7cm}\times \int_{X^p}\left[\prod_{j=1}^{p-1}\prod_{l=j+1}^{p}  \mathbbm{1}_{B(x_j)}(x_l)\right]\left[\prod_{j=1}^p\prod_{l=p+1}^k  \mathbbm{1}_{B(x_j)}(T^{i'_l}x_l)\right]d\mu^p(1,p) d\mu^{k-p}({p+1},k)\nonumber\\
&&\hspace{0.8cm}+c^2pr_n^{-2\xi}\theta_g=:II+c^2pr_n^{-2\xi}\theta_g.\nonumber
\end{eqnarray}

Therefore, when $i_1-i'_1>g$, $i_2-i'_2>g$,\dots, $i_k-i'_k>g$ we have 
\begin{eqnarray}
& &\int_{X^k}\prod_{j=1}^{k-1} \prod_{l=j+1}^k \mathbbm{1}_{B(T^{i_j}x_j)}(T^{i_l}x_l)  \mathbbm{1}_{B(T^{i'_j}x_j)}(T^{i'_l}x_l)d\mu^k(1,k)\nonumber\\
&&\hspace{1cm}\leq c^2kr_n^{-2\xi}\theta_g+\left(\int_{X^k}\prod_{j=1}^{k-1} \prod_{l=j+1}^k \mathbbm{1}_{B(x_j)}(x_l)d\mu^k(1,k)\right)^2.\label{eqmixk}
\end{eqnarray}

Now, if $i_1-i'_1>g$, $i_2-i'_2\leq g$,\dots, $i_k-i'_k\leq g$, we first proceed as in \eqref{eqmix1} and then, to estimate the term $I$, we use H\"older's inequality to find that
\begin{eqnarray}
&I=& \int_{X^{k-1}}\prod_{j=2}^{k-1} \prod_{l=j+1}^k \mathbbm{1}_{B(T^{i_j}x_j)}(T^{i_l}x_l) \left[\int_X \prod_{l=2}^k \mathbbm{1}_{B(x_1)}(T^{i_l}x_l)d\mu(x_1)\right] \nonumber\\
& & \times\prod_{j=2}^{k-1} \prod_{l=j+1}^k\mathbbm{1}_{B(T^{i'_j}x_j)}(T^{i'_l}x_l) \left[\int_X \prod_{l=2}^k \mathbbm{1}_{B(x_1)}(T^{i'_l}x_l)d\mu(x_1) \right]d\mu^{k-1}(2,k)\nonumber\\
&&\leq \left(\int_{X^{k-1}}\left(\prod_{j=2}^{k-1} \prod_{l=j+1}^k \mathbbm{1}_{B(T^{i_j}x_j)}(T^{i_l}x_l) \left[\int_X \prod_{l=2}^k \mathbbm{1}_{B(x_1)}(T^{i_l}x_l)d\mu(x_1)\right]\right)^2 d\mu^{k-1}(2,k)\right)^{1/2}\nonumber\\
& & \times\left(\int_{X^{k-1}}\left(\prod_{j=2}^{k-1} \prod_{l=j+1}^k \mathbbm{1}_{B(T^{i_j}x_j)}(T^{i'_l}x_l) \left[\int_X \prod_{l=2}^k \mathbbm{1}_{B(x_1)}(T^{i'_l}x_l)d\mu(x_1)\right]\right)^2 d\mu^{k-1}(2,k)\right)^{1/2}.\nonumber
\end{eqnarray}
Finally we use the invariance of $\mu$ to conclude that
\begin{eqnarray}
&I\leq& \int_{X^{k-1}}\left(\prod_{j=2}^{k-1} \prod_{l=j+1}^k \mathbbm{1}_{B(x_j)}(x_l) \left[\int_X \prod_{l=2}^k \mathbbm{1}_{B(x_1)}(x_l)d\mu(x_1)\right]\right)^2 d\mu^{k-1}(2,k)\nonumber\\
&=& \int_{X^{k-1}}\prod_{j=2}^{k-1} \prod_{l=j+1}^k \mathbbm{1}_{B(x_j)}(x_l) \left[\int_X \prod_{l=2}^k \mathbbm{1}_{B(x_1)}(x_l)d\mu(x_1)\right]^2 d\mu^{k-1}(2,k).\label{eqfinmix1}
\end{eqnarray}

In the case $i_1-i'_1>g$,\dots,$i_p-i'_p> g$ and $i_{p+1}-i'_{p+1}\leq g$,\dots, $i_k-i'_k\leq g$  we proceed as in \eqref{eqmix1}- \eqref{eqmixp} and then we use Holder's inequality to infer that
\begin{eqnarray*}
&II=&\int_{X^{k-p}} f(p+1,k)g(p+1,k) d\mu^{k-p}(p+1,k)\\
&\leq&\left(\int_{X^{k-p}} f^2(p+1,k)d\mu^{k-p}(p+1,k)\right)^{1/2}\left(\int_{X^{k-p}} g^2(p+1,k)d\mu^{k-p}(p+1,k)\right)^{1/2},\\
\end{eqnarray*}
where $f(p+1,k)$ denotes the function $f(x_{p+1},\dots,x_k)$ defined as $$f(p+1,k)=\prod_{j=p+1}^{k-1} \prod_{l=j+1}^k \mathbbm{1}_{B(T^{i_j}x_j)}(T^{i_l}x_l)\int_{X^p}\prod_{j=1}^{p-1}\prod_{l=j+1}^{p}  \mathbbm{1}_{B(x_j)}(x_l)\prod_{j=1}^p\prod_{l=p+1}^k  \mathbbm{1}_{B(x_j)}(T^{i_l}x_l)d\mu^p(1,p),$$ 
and analogously for $g(p+1,k)=g(x_{p+1},\dots,x_k)$
$$g(p+1,k)=\prod_{j=p+1}^{k-1} \prod_{l=j+1}^k \mathbbm{1}_{B(T^{i'_j}x_j)}(T^{i'_l}x_l)\int_{X^p}\prod_{j=1}^{p-1}\prod_{l=j+1}^{p}  \mathbbm{1}_{B(x_j)}(x_l)\prod_{j=1}^p\prod_{l=p+1}^k  \mathbbm{1}_{B(x_j)}(T^{i'_l}x_l)d\mu^p(1,p).$$

Then  we use the invariance of $\mu$ to infer that
\begin{equation}
II\leq\int_{X^{k-p}}\prod_{j=p+1}^{k-1} \prod_{l=j+1}^k \mathbbm{1}_{B(x_j)}(x_l)\left(\int_{X^p}\prod_{j=1}^{p}\prod_{l=j+1}^{k}  \mathbbm{1}_{B(x_j)}(x_l)d\mu^p(1,p)\right)^2d\mu^{k-p}(p+1,k). \label{eqfinmixp}
\end{equation}
Finally we observe that if $i_1-i'_1\leq g$,\dots, $i_k-i'_k\leq g$ then
\begin{eqnarray}
& &\hspace{-0.5cm}\int_{X^k}\prod_{j=1}^{k-1} \prod_{l=j+1}^k \mathbbm{1}_{B(T^{i_j}x_j)}(T^{i_l}x_l)  \mathbbm{1}_{B(T^{i'_j}x_j)}(T^{i'_l}x_l)d\mu^k(1,k)\leq\int_{X^k}\prod_{j=1}^{k-1} \prod_{l=j+1}^k \mathbbm{1}_{B(T^{i_j}x_j)}(T^{i_l}x_l) d\mu^k(1,k)\nonumber\\
&&\hspace{0.5cm}=\int_{X^k}\prod_{j=1}^{k-1} \prod_{l=j+1}^k \mathbbm{1}_{B(x_j)}(x_l) d\mu^k(1,k)=n^{-k}\E(S_n).\label{eqnomix}
\end{eqnarray}

One can notice that all the other cases can be treated by symmetry.

Thus from \eqref{eqmnvarsn} and  \eqref{eqmixk}- \eqref{eqnomix} we conclude that
\begin{eqnarray*}
& &\mu^k \left(m_n(x_1,\dots,x_k)\geq r_n\right)
\leq\frac{1}{\E (S_n)^2}\Bigg[n^{2k}c^2kr_n^{-2\xi}\theta_g+\sum_{p=1}^{k-1}\binom{k}{p}n^{2p+k-p}g^{k-p}\times\\
&&\hspace{0.5cm}\times\int_{X^{k-p}}\prod_{j=p+1}^{k-1} \prod_{l=j+1}^k \mathbbm{1}_{B(x_j)}(x_l)\left(\int_{X^p}\prod_{j=1}^{p}\prod_{l=j+1}^{k}  \mathbbm{1}_{B(x_j)}(x_l)d\mu^p\right)^2d\mu^{k-p}(p+1,k)\nonumber \\
& &\hspace{0.5cm}+\sum_{p=1}^{k-1}\binom{k}{p}n^{2p+k-p}g^{k-p}c^2pr_n^{-2\xi}\theta_g+g^k \E (S_n)\Bigg].\nonumber
\end{eqnarray*}
Thus, by Lemma \ref{lema1}, we deduce
\begin{eqnarray*}
& &\mu^k \left(m_n(x_1,\dots,x_k)\geq r_n\right)
\leq\frac{1}{\E (S_n)^2}\Bigg[n^{2k}c^2kr_n^{-2\xi}\theta_g+\sum_{p=1}^{k-1}\binom{k}{p}n^{p+k}g^{k-p}c^2pr_n^{-2\xi}\theta_g\nonumber\\
& &+\sum_{p=1}^{k-1}\binom{k}{p}n^{p+k}g^{k-p}\left(n^{-k}\E(S_n)\right)^{(p+k)/k}
+g^k \E (S_n)\Bigg]\nonumber\\
&&=\theta_g\frac{n^{2k}c^2kr_n^{-2\xi}+\sum_{p=1}^{k-1}\binom{k}{p}n^{p+k}g^{k-p}c^2pr_n^{-2\xi}}{\E (S_n)^2}+\sum_{p=1}^{k-1}\frac{\binom{k}{p}g^{k-p}}{\E(S_n)^{2-(p+k)/k}}+\frac{g^k}{ \E (S_n)}.
\end{eqnarray*}
By definitions of $r_n$, $k_n$, \eqref{eqexp} and Lemma~\ref{lemmagendim}, we observe that for $n$ large enough we have
\[\E(S_n)\geq(\log n)^{-b}.\]
Since $g=\log\left(n^\gamma\right)$ we have for  $\gamma$ large enough  that
\[\theta_g\frac{n^{2k}c^2kr_n^{-2\xi}+\sum_{p=1}^{k-1}\binom{k}{p}n^{p+k}g^{k-p}c^2pr_n^{-2\xi}}{\E (S_n)^2}=\mathcal{O}\left(\frac{1}{\log n}\right).\]
Then, we can choose $b\ll-1$ such that 
\[\sum_{p=1}^{k-1}\frac{\binom{k}{p}g^{k-p}}{\E(S_n)^{2-(p+k)/k}}+\frac{g^k}{ \E (S_n)}\leq\sum_{p=1}^{k-1}\frac{\binom{k}{p}g^{k-p}}{(\log n)^{-b(k-p)/k}}+\frac{g^k}{ (\log n)^{-b}}=\mathcal{O}\left(\frac{1}{\log n}\right),\]
and we have
\begin{equation}\label{eqfindsup}
\mu^k \left(m_n(x_1,\dots,x_k)\geq r_n\right)=\mathcal{O}\left(\frac{1}{\log n}\right).
\end{equation}
To conclude the proof it suffices to take a subsequence $n_\ell$ and use Borel-Cantelli Lemma as in the proof of \eqref{eqren1}.
\end{proof}


In order to simplify the proof of Theorem~\ref{thdsuphold}, we state and prove  the following technical lemma:
\begin{lemma}\label{lemmaHA} 
Let $\varphi$ be given by
\[ \varphi(x_p)=\int_{X^{p-1}}\left[\prod_{j=1}^{p-1}\prod_{l=j+1}^{k}  \mathbbm{1}_{B(x_j)}(x_l)\right]d\mu^{p-1}(1,{p-1}).\]
and suppose that (HA) is satisfied. Then there exist $0< r_0<1$, $c>0$ and $\zeta \geq 0$ such that for every $p\in\{2,\dots,k\}$, for $\mu^{k-p}$-almost every $x_{p+1},\dots,x_k\in X$ and for any $0<r<r_0$, the function $\varphi$ belongs to $\mathcal{H}^{\alpha}(X,\R)$ and
$$
||\varphi||_{\mathcal{H}^{\alpha}} 
\leq cr^{-\zeta}.
$$

\end{lemma}

\begin{proof}

Let $0<r<r_0$ and $x,y\in X$.  We have
\begin{eqnarray*}
& &|\varphi(x)-\varphi(y)|=\left|\int_{X^{p-1}}\prod_{j=1}^{p-1}\prod_{\substack{l=j+1\\ l\neq p}}^{k}  \mathbbm{1}_{B(x_j)}(x_l)\left[\prod_{j=1}^{p-1}  \mathbbm{1}_{B(x_j)}(x)-\prod_{j=1}^{p-1}  \mathbbm{1}_{B(x_j)}(y)\right]d\mu^{p-1}(1,{p-1})\right|\\
&&\hspace{0.5cm}\leq\int_{X^{p-1}}\left|\prod_{j=1}^{p-1}  \mathbbm{1}_{B(x)}(x_j)-\prod_{j=1}^{p-1}  \mathbbm{1}_{B(y)}(x_j)\right|d\mu^{p-1}(1,{p-1})\\
&&\hspace{0.5cm}\leq\sum_{l=0}^{p-2}\int_{X^{p-1}}\left|\prod_{j=1}^{l}  \mathbbm{1}_{B(y)}(x_j)\prod_{j=l+1}^{p-1}  \mathbbm{1}_{B(x)}(x_j)-\prod_{j=1}^{l+1}  \mathbbm{1}_{B(y)}(x_j)\prod_{j=l+2}^{p-1}  \mathbbm{1}_{B(x)}(x_j)\right|d\mu^{p-1}(1,{p-1})\\
&&\hspace{0.5cm}\leq\sum_{l=0}^{p-2}\int_{X}\left|  \mathbbm{1}_{B(x)}(x_{l+1})-  \mathbbm{1}_{B(y)}(x_{l+1})\right|d\mu(x_{l+1})
=(p-1)\int_{X}\left|  \mathbbm{1}_{B(x)}(z)-  \mathbbm{1}_{B(y)}(z)\right|d\mu(z).
\end{eqnarray*}

If $d(x,y)\geq r $ then
\begin{equation}\label{phihold1}
 |\varphi(x)-\varphi(y)|\leq 2(p-1)
\leq\frac{2(p-1)}{r}d(x,y).
\end{equation}

If otherwise $d(x,y)<r $ then from (HA) we conclude that
\begin{eqnarray}
|\varphi(x)-\varphi(y)|&\leq&(p-1)\int_{X}\left|  \mathbbm{1}_{B(x)}(z)-  \mathbbm{1}_{B(y)}(z)\right|d\mu(z)\label{phihold2}\\
&&\leq(p-1)  \left[\mu\left(B(x,r)\backslash \left(B(x,r)\cap B(y,r)\right)\right)+\mu\left(B(y,r)\backslash \left(B(x,r)\cap B(y,r)\right)\right)\right]\nonumber\\
&&\leq(p-1)  \mu\left(B(x,r+d(x,y))\backslash B(x,r-d(x,y))\right)\leq(p-1)r^{-\xi}d(x,y)^{\beta},\nonumber
\end{eqnarray}

and the lemma follows from inequalities \eqref{phihold1} and \eqref{phihold2}.

\end{proof}

\begin{proof}[Proof of Theorem~\ref{thdsuphold}]
When our Banach space $\mathcal{C}$ is the space of H\"older functions, (H2) cannot be satisfied since characteristic functions are not continuous. Thus, we need to adapt the proof of Theorem~\ref{dsup} to this setting, approximating characteristic functions by Lipschitz functions, following the construction of the proof of Lemma 9 in \cite{Saussol1}.  We will only prove the key part here, which is obtaining the equivalent of our inequality \eqref{eqmixp}.

To do so we fix $q\in\{1,\dots,p\}$ and consider the the following term: 
\begin{eqnarray*}
& &I:= \int_{X}\big[\prod_{l=q+1}^k \mathbbm{1}_{B(T^{i_{q}}x_{q})}(T^{i_l}x_l)  \mathbbm{1}_{B(T^{i'_{q}}x_{q})}(T^{i'_l}x_l)\nonumber\\
& &\hspace{0.5cm} \times\int_{X^{q-1}}\left[\prod_{j=1}^{q-2}\prod_{l=j+1}^{q-1}  \mathbbm{1}_{B(x_j)}(x_l)\right]\left[\prod_{j=1}^{q-1}\prod_{l=q}^k  \mathbbm{1}_{B(x_j)}(T^{i_l}x_l)\right]d\mu^{q-1}(1,{q-1}) \\
& & \hspace{0.5cm}\times\int_{X^{q-1}}\left[\prod_{j=1}^{q-2}\prod_{l=j+1}^{q-1}  \mathbbm{1}_{B(x_j)}(x_l)\right]\left[\prod_{j=1}^{q-1}\prod_{l=q}^k  \mathbbm{1}_{B(x_j)}(T^{i'_l}x_l)\right]d\mu^{q-1}(1,{q-1}) \big]d\mu(x_{q}).\nonumber
\end{eqnarray*}
We assume $|i_q-i_q'|>g$. 
Let $\rho>0$ (to be choosen later). Let $\eta_{r_n}:[0, \infty) \to \R$ be the $\frac{1}{{\rho r_n}}$-Lipschitz function such that $\mathbbm{1}_{[0,{r_n}]} \leq \eta_{r_n} \leq \mathbbm{1}_{[0,(1+\rho){r_n}]}$ and set
 \begin{equation}\label{defvarphi}
 \varphi_{x_{q+1},...,x_k,{r_n}}(x)= \prod_{l=q+1}^{k} \eta_{r_n}(d(x,x_l)).
 \end{equation}
 We observe that $\varphi_{x_{q+1},...,x_k,{r_n}}$ is $\frac{k-q}{\rho r_n}$-Lipschitz. Moreover, we have
\begin{equation}\label{eqapproxlip}
\prod_{l=q+1}^k \mathbbm{1}_{B(x,r_n)}(T^{i_l}x_l) \leq \varphi_{T^{i_{q+1}}x_{q+1},...,T^{i_{k}}x_k,{r_n}}(x)\leq \prod_{l=q+1}^k \mathbbm{1}_{B(x,(1+\rho)r_n)}(T^{i_l}x_l). 
\end{equation}
Now we define the following auxiliary function
\begin{eqnarray}
& &\Phi_{T^{i_{q+1}}x_{q+1},...,T^{i_{k}}x_k,{r}}(x_q)=\varphi_{T^{i_{q+1}}x_{q+1},...,T^{i_{k}}x_k,{r}}(x_q)\times\label{bigphi}\\
&&\hspace{0.5cm}\times\int_{X^{q-1}}\left[\prod_{j=1}^{q-2}\prod_{l=j+1}^{q}  \mathbbm{1}_{B(x_j)}(x_l)\right]\left[\prod_{j=1}^{q-1}\prod_{l=q+1}^k  \mathbbm{1}_{B(x_j)}(T^{i_l}x_l)\right]d\mu^{q-1}(1,q-1).\nonumber
\end{eqnarray}
From Lemma \ref{lemmaHA}, we observe that for $\mu^{k-q+1}$-almost every $x_{q+1},\dots,x_k\in X$ and for any $0<r<r_0$ the function $\Phi$ belongs to $\mathcal{H}^{\alpha}(X,\R)$ and
$$
||\Phi_{T^{i_{q+1}}x_{q+1},...,T^{i_{k}}x_k,{r}}||_{\mathcal{H}^\alpha} \leq cr^{-\zeta}+(k-q)(\rho r)^{-1}\leq cr^{-\zeta}+k(\rho r)^{-1}.
$$
Using (H1), \eqref{eqapproxlip} and \eqref{bigphi} we deduce that
 \begin{eqnarray}
&I\leq &\int_{X}\Phi_{T^{i_{q+1}}x_{q+1},...,T^{i_{k}}x_k,{r_n}}(T^{i_q}x_q)\Phi_{T^{i'_{q+1}}x_{q+1},...,T^{i'_{k}}x_k,{r_n}}(T^{i'_{q}}x_q)d\mu(x_q)\nonumber\\
&\leq &\int_{X}\Phi_{T^{i_{q+1}}x_{q+1},...,T^{i_{k}}x_k,{r_n}}(x_q)d\mu(x_q)\int_{X}\Phi_{T^{i'_{q+1}}x_{q+1},...,T^{i'_{k}}x_k,{r_n}}(x_q)d\mu(x_q)\nonumber\\
& +& \theta_g \lt\Vert\Phi_{T^{i_{q+1}}x_{q+1},...,T^{i_{k}}x_k,{r_n}}\rt\Vert_{\mathcal{H}^\alpha} \lt\Vert \Phi_{T^{i'_{q+1}}x_{q+1},...,T^{i'_{k}}x_k,{r_n}} \rt\Vert_{\mathcal{H}^\alpha} \nonumber\\
&\leq &\int_{X}\Phi_{T^{i_{q+1}}x_{q+1},...,T^{i_{k}}x_k,{r_n}}(x_q)d\mu(x_q)\int_{X}\Phi_{T^{i'_{q+1}}x_{q+1},...,T^{i'_{k}}x_k,{r_n}}(x_q)d\mu(x_q)\nonumber\\
& &+(cr_n^{-\zeta}+k(\rho r_n)^{-1})^2\theta_g.\label{bigphimix}
\end{eqnarray}

At this point we observe that the inequality \eqref{bigphimix} together with \eqref{eqapproxlip} will not be sufficient to obtain our equivalent of \eqref{eqmixp} since the radius of the balls will be $(1+\rho)r_n$ (instead of $r_n$). To overcome this problem we use (HA). For simplicity, we use the following notation:
\[g(x_q)=\int_{X^{q-1}}\left[\prod_{j=1}^{q-2}\prod_{l=j+1}^{q}  \mathbbm{1}_{B(x_j)}(x_l)\right]\left[\prod_{j=1}^{q-1}\prod_{l=q+1}^k  \mathbbm{1}_{B(x_j)}(T^{i_l}x_l)\right]d\mu^{q-1}(1,{q-1}).\]
Thus from \eqref{eqapproxlip} we know that
\begin{eqnarray}
& &\int_{X}\Phi_{T^{i_{q+1}}x_{q+1},...,T^{i_{k}}x_k,{r_n}}(x_q)d\mu(x_q)=\int_X \varphi_{T^{i_{q+1}}x_{q+1},...,T^{i_{k}}x_k,{r}}(x_q)g(x_p)d\mu(x_q)\label{ineqbigphiha}\\
&&\leq\int_X g(x_q)\prod_{l=q+1}^k \mathbbm{1}_{B(x_q,(1+\rho)r_n)}(T^{i_l}x_l)d\mu(x_q)=\int_X g(x_q)\prod_{l=q+1}^k \mathbbm{1}_{B(x_q,r_n)}(T^{i_l}x_l)d\mu(x_q)\nonumber\\
& &+\int_X g(x_q)\left(\prod_{l=q+1}^k \mathbbm{1}_{B(x_q,(1+\rho)r_n)}(T^{i_l}x_l)-\prod_{l=q+1}^k \mathbbm{1}_{B(x_q,r_n)}(T^{i_l}x_l)\right)d\mu(x_q)\nonumber\\
&&\leq\int_X g(x_q)\prod_{l=q+1}^k \mathbbm{1}_{B(x_q,r_n)}(T^{i_l}x_l)d\mu(x_q)+\mu\left(\bigcap_{l=q+1}^k B(T^{i_l}x_l,(1+\rho)r_n)\setminus \bigcap_{l=q+1}^k B(T^{i_l}x_l,r_n)\right).\nonumber
\end{eqnarray}

Using (HA) we conclude that
\begin{eqnarray}
& &\mu\left(\bigcap_{l=q+1}^k B(T^{i_l}x_l,(1+\rho)r_n)\setminus \bigcap_{l=q+1}^k B(T^{i_l}x_l,r_n)\right)\leq \mu\left(\bigcup_{l=q+1}^k B(T^{i_l}x_l,(1+\rho)r_n)\setminus B(T^{i_l}x_l,r_n)\right)\nonumber\\
& &\leq\sum_{l=q+1}^k\mu\left( B(T^{i_l}x_l,(1+\rho)r_n)\setminus B(T^{i_l}x_l,r_n)\right)\leq(k-q)r_n^{-\xi}\rho^{\beta}.\label{ineqbigphiha2}
\end{eqnarray}

Then from \eqref{ineqbigphiha} and \eqref{ineqbigphiha2} and  taking  $\rho$  small enough we infer that
\begin{eqnarray}
& &\int_{X}\Phi_{T^{i_{q+1}}x_{q+1},...,T^{i_{k}}x_k,{r_n}}(x_q)d\mu(x_q)\int_{X}\Phi_{T^{i'_{q+1}}x_{q+1},...,T^{i'_{k}}x_k,{r_n}}(x_q)d\mu(x_q)\nonumber\\
&\leq & \int_{X}\prod_{l=q+1}^k \mathbbm{1}_{B(x_{q})}(T^{i_l}x_l) \int_{X^{q-1}}\prod_{j=1}^{q-2}\prod_{l=j+1}^{q}  \mathbbm{1}_{B(x_j)}(x_l)\prod_{j=1}^{q-1}\prod_{l=q+1}^k  \mathbbm{1}_{B(x_j)}(T^{i_l}x_l)d\mu^{q-1}(1,q-1)d\mu(x_{q}) \nonumber\\
& &\!\!\!\times \int_{X}\prod_{l=q+1}^k \mathbbm{1}_{B(x_{q})}(T^{i'_l}x_l)\int_{X^{q-1}}\prod_{j=1}^{q-2}\prod_{l=j+1}^{q}  \mathbbm{1}_{B(x_j)}(x_l)\prod_{j=1}^{q-1}\prod_{l=q+1}^k  \mathbbm{1}_{B(x_j)}(T^{i'_l}x_l)d\mu^{q-1}(1,q-1) d\mu(x_{q})\nonumber\\
& &+3(k-q)r_n^{-\xi}\rho^{\beta}\nonumber\\
&= &  \int_{X^{q}}\left[\prod_{j=1}^{q-2}\prod_{l=j+1}^{q}  \mathbbm{1}_{B(x_j)}(x_l)\right]\left[\prod_{j=1}^{q}\prod_{l=q+1}^k  \mathbbm{1}_{B(x_j)}(T^{i_l}x_l)\right]d\mu^{q}(1,q) \nonumber\\
& &\times\int_{X^{q}}\left[\prod_{j=1}^{q-2}\prod_{l=j+1}^{q}  \mathbbm{1}_{B(x_j)}(x_l)\right]\left[\prod_{j=1}^{q}\prod_{l=q+1}^k  \mathbbm{1}_{B(x_j)}(T^{i'_l}x_l)\right]d\mu^{q}(1,q) +3(k-q)r_n^{-\xi}\rho^{\beta}\label{ineqbigphifin}.
\end{eqnarray}
Thus, since $q\leq k$, \eqref{ineqbigphifin} together with \eqref{bigphimix} gives us
\begin{eqnarray*}
&I\leq&  \int_{X^{q}}\left[\prod_{j=1}^{q-2}\prod_{l=j+1}^{q}  \mathbbm{1}_{B(x_j)}(x_l)\right]\left[\prod_{j=1}^{q}\prod_{l=q+1}^k  \mathbbm{1}_{B(x_j)}(T^{i_l}x_l)\right]d\mu^{q}(1,{q}) \nonumber\\
& &\times\int_{X^{q}}\left[\prod_{j=1}^{q-2}\prod_{l=j+1}^{q}  \mathbbm{1}_{B(x_j)}(x_l)\right]\left[\prod_{j=1}^{q}\prod_{l=q+1}^k  \mathbbm{1}_{B(x_j)}(T^{i'_l}x_l)\right]d\mu^{q}(1,{q})\nonumber\\
& &+3kr_n^{-\xi}\rho^{\beta}+(cr_n^{-\zeta}+k(\rho r_n)^{-1})^2\theta_g.\label{bigphimix2}
\end{eqnarray*}

Repeating this process for each $q\in\{1,\dots,p\}$ we  obtain our equivalent of \eqref{eqmixp}
\begin{eqnarray}
& &\int_{X^k}\prod_{j=1}^{k-1} \prod_{l=j+1}^k \mathbbm{1}_{B(T^{i_j}x_j)}(T^{i_l}x_l)  \mathbbm{1}_{B(T^{i'_j}x_j)}(T^{i'_l}x_l)d\mu^k(1,k)\nonumber\\
&&\leq3pkr_n^{-\xi}\rho^{\beta}+p(cr_n^{-\zeta}+k(\rho r_n)^{-1})^2\theta_g+ \int_{X^{k-p}}\prod_{j=p+1}^{k-1} \prod_{l=j+1}^k \mathbbm{1}_{B(T^{i_j}x_j)}(T^{i_l}x_l)  \mathbbm{1}_{B(T^{i'_j}x_j)}(T^{i'_l}x_l)\nonumber\\
&  \times&\int_{X^p}\left[\prod_{j=1}^{p-1}\prod_{l=j+1}^{p}  \mathbbm{1}_{B(x_j)}(x_l)\right]\left[\prod_{j=1}^p\prod_{l=p+1}^k  \mathbbm{1}_{B(x_j)}(T^{i_l}x_l)\right]d\mu^p(1,p) \nonumber\\
&  \times&\int_{X^p}\left[\prod_{j=1}^{p-1}\prod_{l=j+1}^{p}  \mathbbm{1}_{B(x_j)}(x_l)\right]\left[\prod_{j=1}^p\prod_{l=p+1}^k  \mathbbm{1}_{B(x_j)}(T^{i'_l}x_l)\right]d\mu^p(1,p) d\mu^{k-p}(p+1k).\nonumber
\end{eqnarray}
Thus, the rest of the proof follows exactly as in Theorem~\ref{dsup} where at the end, one must choose $\rho=n^{-\delta}$ with $\delta$ large enough so that \eqref{eqfindsup} holds.

\end{proof}

\begin{proof}[Proof of Theorem~\ref{thprincobs}]
The proof follows the lines of the proof of Theorems~\ref{thineq} and \ref{thdsuphold}, replacing $S_n$ by
\begin{equation*}
S_n^f(x_1,\dots,x_k)=\sum_{i_1,\dots,i_k =0}^{n-1}\prod_{j=1}^{k-1} \prod_{l=j+1}^k \mathbbm{1}_{f^{-1}B(f(T^{i_j}x_j),r_n)}(T^{i_l}x_l).
\end{equation*}
Another modification worth mentioning is that in \eqref{defvarphi}, $\varphi_{x_{q+1},...,x_k,{r_n}}(x)$ must be replaced by
\begin{equation*}
 \varphi^f_{x_{q+1},...,x_k,{r_n}}(x)= \prod_{l=q+1}^{k} \eta_{r_n}(d(f(x),f(x_l))),
 \end{equation*}
which is a $\frac{L(k-q)}{\rho r_n}$-Lipschitz function (if $f$ is $L$-Lipschitz).
\end{proof}


\subsection*{Acknowledgements}The authors would like to thank Benoit Saussol for useful discussions.


\end{document}